\documentclass[12pt]{article}
\usepackage{graphics,epsfig,amssymb}
\usepackage{amsmath}
\usepackage{amsthm}
\usepackage{amsfonts}
\usepackage{amssymb,latexsym}
\usepackage[all]{xy}
\usepackage{txfonts}
\usepackage{amssymb}
\usepackage{graphicx}
\usepackage{amsfonts}
\usepackage{verbatim}
\usepackage{mathrsfs}
\usepackage{color}
\usepackage{hyperref}
\newtheorem{Theorem}{Theorem}[section]
\newtheorem{Proposition}[Theorem]{Proposition}
\newtheorem{Lemma}[Theorem]{Lemma}
\newtheorem{Remark}[Theorem]{Remark}

\def\QEDopen{{\setlength{\fboxsep}{0pt}\setlength{\fboxrule}{0.2pt}\fbox{\rule[0pt]{0pt}{1.3ex}\rule[0pt]{1.3ex}{0pt}}}}

\begin{document}

\vskip -1.5cm

\title{On nodal solutions of a nonlocal Choquard equation in a bounded domain}

\author{Changfeng Gui\thanks{Corresponding author},\quad
Hui Guo\thanks{Supported partially by the
National Natural Science Foundation of P. R. China (Grant No.11371128) }}

\maketitle

\vskip -5.5cm \small{
\begin{abstract}
In this paper, we are interested in the least energy nodal solutions to the following nonlocal Choquard equation with a local term
\begin{equation*}\left\{\begin{array}{rll}
-\Delta u&=\lambda|u|^{p-2}u+\mu \phi(x)|u|^{q-2}u\\
-\Delta \phi&=|u|^q\\
u&=\phi=0
\end{array}\right.
\begin{gathered}\begin{array}{rll}
&\mbox{in}\ \Omega,\\
&\mbox{in}\ \Omega,\\
&\mbox{on}\ \partial\Omega,
\end{array}\end{gathered}\end{equation*}
where $ \lambda,\mu>0, p\in [2,6), q\in (1,5)$ and $\Omega\subset \mathbb{R}^3$ is a bounded domain.  This problem may be seen as a nonlocal perturbation of the classical Lane-Emden equation $-\Delta u=\lambda|u|^{p-2}u$ in $\Omega.$  The problem has a variational functional with a nonlocal term $\mu\int_{\Omega}\phi|u|^q$.  The appearance of the nonlocal term  makes the variational functional very
 different from the local case $\mu=0$, for which the problem has ground state solutions and least energy nodal solutions if $p\in (2,6)$. The problem may also be
 viewed as a nonlocal  Choquard equation  with a local pertubation term when $\lambda \not =0$. For $\mu>0$,  we show that although ground state solutions always exist, the existence of  least energy
 nodal solution depends on $q$:  for  $q\in(1,2)$ there does not exist a least energy nodal solution  while for $q\in[2,5)$  such a solution exists.
 Note that $q=2$ is a critical value.  In the case of a linear
 local perturbation, i.e., $p=2,$ if $\lambda<\lambda_1,$ the problem has a positive ground state and a least energy nodal solution. However,  if $\lambda\geq \lambda_1,$  the problem has a ground state which changes sign. Hence it is also a least energy nodal solution.

\end{abstract}}


\section{Introduction}
In this paper, we mainly study the following system
\begin{equation}\label{eq1.1}\left\{\begin{array}{rll}
-\Delta u&=\lambda|u|^{p-2}u+\mu \phi(x)|u|^{q-2}u\\
-\Delta \phi&=|u|^q\\
u&=\phi=0
\end{array}\right.
\begin{gathered}\begin{array}{rll}
&\mbox{in}\ \Omega,\\
&\mbox{in}\ \Omega,\\
&\mbox{on}\ \partial\Omega,
\end{array}\end{gathered}\end{equation}
where $p\in [2,6), q\in (1,5),\lambda,\mu>0$ and $\Omega\subset \mathbb{R}^3$ is a bounded domain.

When $\mu=0,$ system \eqref{eq1.1} is reduced to
\begin{equation}\label{eq1.2}
-\Delta u=|u|^{p-2}u\ \mbox{in $\Omega$},\ u=0\ \mbox{on $\partial\Omega$}.
\end{equation}
Equation \eqref{eq1.2} has been widely studied in the past decades and there are  many  results of \eqref{eq1.2} in the literature. In particular, the existence of ground state solutions and  least energy nodal solutions (sign-changing solutions) for $p\in(2,6)$ were considered in \cite{Bartsch-Weth-Willem,Castro-Cossio-Neuberger,Rabinowitz}. Here a least energy nodal solutions means a critical point of the associated energy functional  attaining  the infimum  under a constraint on Nehari nodal set (to be defined later).  For more results regarding \eqref{eq1.2},  one may refer to \cite{Rabinowitz,Willem,Zou} and references therein.

If $\mu<0$ and $q=2,$ system \eqref{eq1.1} becomes Schr\"odinger-Poisson equation in a bounded domain
\begin{equation}\label{eq1.1*}\left\{\begin{array}{rll}
-\Delta u+|\mu| \phi(x) u&=|u|^{p-2}u\\
-\Delta \phi&=|u|^2\\
u&=\phi=0
\end{array}\right.
\begin{gathered}\begin{array}{rll}
&\mbox{in}\ \Omega,\\
&\mbox{in}\ \Omega,\\
&\mbox{on}\ \partial\Omega.
\end{array}\end{gathered}\end{equation}
In this case, Ruiz and Siciliano \cite{Ruiz-Siciliano} proved that if $p\in(3,6),$ equation \eqref{eq1.1*} has one solution for almost every $|\mu|>0$; if $p=3,$  equation \eqref{eq1.1*} has one solution for small $|\mu|>0$ but no solutions for large $|\mu|>0;$ while $p\in(2,3),$  equation \eqref{eq1.1*} has two solutions for small $|\mu|>0$ but no solutions for large $|\mu|>0.$  These results show that $p=3$ is a critical value. Moreover, Alves and Souto \cite{Alves-Souto-nodalsolution} obtained a least energy nodal solution for general nonlinearity $f(u)$ instead of $|u|^{p-2}u$ when $|\mu|=1$, and Batkam \cite{Batkam} found infinitely many high energy nodal solutions. Similar problems for $\Omega=\mathbb{R}^3$ have also been  studied, including  the following Schrodinger-Poisson equation
\begin{equation}\label{eq1.3}
-\Delta u+u+|\mu| \phi u=|u|^{p-2}u\quad\mbox{in}\ \mathbb{R}^3.
\end{equation}
In \cite{Ambrosetti-Ruiz,Azzollini-Pomponio,Ruiz},  the existence of ground state solutions and multiple solutions are shown. Moreover, via the method of Nehari nodal set, radial nodal solutions of \eqref{eq1.3} with  $p\in (4,6)$  are obtained in \cite{Kim-Seok,Liu-Wang-Zhang}, while nonexistence of least energy nodal solutions are shown in \cite{Guo}. By applying the invariant sets of descending flow, Wang and Zhou \cite{Wang-Zhou} proved infinitely many nodal solutions of \eqref{eq1.3} with a suitable potential term $V(x)u$ for $p\in (3,6)$. For more information regarding Schr\"odinger-Poisson equation, the reader can see \cite{Alves-Souto-nodalsolution,Ambrosetti-Ruiz,Azzollini-Pomponio,Batkam,Kim-Seok,
Liu-Wang-Zhang,Pisani-Siciliano,Ruiz,Ruiz-Siciliano,Wang-Zhou}.

When $\mu>0$ and $\Omega=\mathbb{R}^N$, the following related equation \eqref{eq1.1}
\begin{equation}\label{quankj}
-\Delta u+u=\left(\int_{\mathbb{R}^N}\frac{|u(y)|^q}{|x-y|^{N-\alpha}}dy\right)|u|^{q-2}u+|u|^{p-2}u\quad\mbox{in}\ \mathbb{R}^N,
\end{equation}
were considered in \cite{Ao,Chen-Guo}, where $\alpha\in(0,N).$ Precisely, when
 $N=3,\alpha\in(2,3), p=2$ and $4\leq q<6,$ the existence of solutions was obtained in \cite{Chen-Guo}. When either $N=4,\alpha\in(0,N),q\in (3,4)$ or $N\geq 5,\alpha\in(0,N),q\in (2,2^*=\frac{2N}{N-2})$, one nontrivial solution of \eqref{quankj} at $p=\frac{N+\alpha}{N-2}$ was obtained in \cite{Ao}.

The above equation is a modified model of  Choquard equation
\begin{equation}\label{eq1.4}
-\Delta u+u=\left(\int_{\mathbb{R}^N}\frac{|u(y)|^q}{|x-y|^{N-\alpha}}\right)|u|^{q-2}u\quad\mbox{in}\ \mathbb{R}^N,
\end{equation}
which was well studied recently.  Choquard equation (also called Schr\"odinger-Newton system) appears in various physical models. It was proposed by Pekar in 1954 for describing the quantum mechanics of a polaron, and derived by Choquard for describing an electron trapped in its own hole and by Penrose for selfgravitating matter. One may refer to \cite{Choquard-Stubbe-Vuffray,Tod-Moroz,Wei-Winter} and references therein for more information and details. By via the  odd Nehari manifold and Nehari nodal sets (see the exact definition later) and  the minimax principle,  the existence of least energy nodal solutions for $q\in[2,\frac{N+\alpha}{N-2})$  and nonexistence of such solutions for $q\in(1,2)$  are shown in  \cite{Ghimenti-Moroz,Ghimenti-Shaftingen-JFA2016},  with  $q=2$ being  a critical value. The interested reader can refer to  \cite{Ghimenti-Moroz,Ghimenti-Shaftingen-JFA2016,Ye-hongyukirchhoff-choquard}  for some recent results on least energy nodal solutions in $\mathbb{R}^3,$ and \cite{Moroz-Schaftingen, Moroz-Schaftingen-guide} for more related results about Choquard equation.

On the other hand, there are few results about \eqref{eq1.1} in bounded domain. In \cite{Azzollini-D'Avenia}, by using a truncation argument and monotonicity trick, the authors show  that \eqref{eq1.1} has a mountain-pass solution for $\mu $ small enough.  However, it is unknown whether  ground state solutions or least energy nodal solutions of \eqref{eq1.1} exist or not.

In this paper, we shall give a quite complete answer to the above question.

Regarding gound state solutions, we shall show
\begin{Theorem}\label{prop}
Assume that $p\in (2,6)$ and $q\in (1,5).$ Then for any $\lambda,\mu>0,$
 problem \eqref{eq1.1} has a ground state solution, and any ground state solution  must not change sign in $\Omega.$
\end{Theorem}

Regarding nodal solutions, we have
\begin{Theorem}\label{theoremmain}
Assume that $p\in(2,6).$ Then the following statements are true:
\item{(i)} if $2\leq q<5,$ then for any $\lambda,\mu>0,$ problem \eqref{eq1.1} possesses at least one least energy nodal solution;
\item{(ii)} if $1<q<2,$ then for any $\lambda,\mu>0,$ problem \eqref{eq1.1} has no least energy nodal solutions.
\end{Theorem}
Furthermore, when $p=2,$ we have the following results.

\begin{Theorem}\label{zhdl-1}
 Let $\lambda_1$ be the first eigenvalue of $-\Delta.$ Suppose that $\lambda\in(-\infty,\lambda_1)$ and $p=2$.
\item{(i)} If $2\leq q<5,$ then for any $\mu>0,$  problem \eqref{eq1.1} has one ground state and one least energy nodal solution;
\item{(ii)} If $1<q<2,$ then for any $\mu>0,$  problem \eqref{eq1.1} has one ground state but no  least energy nodal solution.
\end{Theorem}

The following result shows an interesting phenomenon that  a ground state solution
can change sign in some case.

\begin{Theorem}\label{zhdl-2}
Suppose that $\lambda\in[\lambda_1,+\infty)$ and $p=2,$ then for any $\mu>0$ and $q\in(1,5),$ problem \eqref{eq1.1} has a ground state solution. Moreover, this solution is a nodal solution.
\end{Theorem}

As usual, we transform  system \eqref{eq1.1} into a single nonlocal equation of $u$ by using a classical reduction approach. This single nonlocal equation has a  variational structure, so we shall use variational methods to obtain the existence of  ground state solutions of \eqref{eq1.1} with $p\in (2,6)$ via Nehari manifold. To find nodal solutions to \eqref{eq1.1}, we note that existing  methods such as those in  \cite{Bartsch-Weth-Willem,Castro-Cossio-Neuberger} do not apply directly to \eqref{eq1.1}, because they rely on the local features of the equations. Moreover, the arguments used for Schrodinger-Poisson system in \cite{Alves-Souto-nodalsolution,Batkam} are  not applicable for \eqref{eq1.1} neither, because they depend on a special property of Nehari nodal set  which the functional of \eqref{eq1.1} does not have. In order to prove our results, we have to develop a new method, based partially on the ideas of \cite{Ghimenti-Moroz,Ghimenti-Shaftingen-JFA2016,Ye-hongyukirchhoff-choquard} on problems in the entire space. The main difficulty in our method lies in  the fact that the nonlocal term  can not be written explicitly as a  convolution type as in \cite{Ghimenti-Moroz,Ghimenti-Shaftingen-JFA2016,Ye-hongyukirchhoff-choquard}.

When $p=2$,  if $\lambda\in(-\infty,\lambda_1),$  problem \eqref{eq1.1} can be handled as in the case $p\in(2,6).$  However,  if $\lambda\in[\lambda_1,+\infty),$
the energy functional is indefinite  and hence can not be treated as before. In order to obtain our results, we deal with this additional difficulty  by  using the method of generalized Nehari manifold (see \cite{Szulkin-Weth}).

The paper is organized as follows. In section 2, we introduce preliminaries and notations. In section 3, we prove Theorem \ref{prop} by using Nehari method, and in section 4, we shall distinguish three cases to prove Theorem \ref{theoremmain}.

\section{Preliminaries}

We first notet that the arguments of this paper are applicable for all $\mu>0.$  Without loss of generality, we may assume $\mu=1$ throughout the paper.

System \eqref{eq1.1} is variational and the corresponding energy functional $J_q:H_0^1(\Omega)\times H_0^1(\Omega)\to \mathbb{R}$ is given by
\begin{equation*}
J_q(u,\phi)=\frac{1}{2}\int_{\Omega} |\nabla u|^2 +\frac{1}{2q}\int_{\Omega} |\nabla \phi|^2 -\frac{1}{q}\int_{\Omega}|u|^q \phi-\frac{\lambda}{p}\int_{\Omega} |u|^{p}.
\end{equation*}
Recall that for any given $u\in H_0^1(\Omega)$ and $q\in (1,5),$ by Lax-Milgram Theorem, there exists a unique $\phi_u\in H_0^1(\Omega)$ such that
$$-\Delta \phi_u=|u|^q.$$
Then it follows immediately that
$$0\leq \int_{\Omega} |\nabla \phi_u|^2=\int_{\Omega}|u|^q \phi_u.$$
This allows us to define one-variable functional
$I_q: H_0^1(\Omega)\to \mathbb{R}$ by
\begin{equation*}
I_q(u):=J_q(u,\phi_u)=\frac{1}{2}\int_{\Omega} |\nabla u|^2 -\frac{1}{2q}\int_{\Omega}|u|^q \phi_u-\frac{\lambda}{p}\int_{\Omega} |u|^{p}.
\end{equation*}
It is easy to check that $I_q\in C^1(H_0^1(\Omega),\mathbb{R}),$ whose Gateaux derivative
is defined by
$$I_q'(u)v=\int_{\Omega} \nabla u\nabla v -\int_{\Omega}\phi_u |u|^{q-2}uv -\lambda\int_{\Omega} |u|^{p-2}uv$$
for all $v\in H_0^1(\Omega).$
As stated  in \cite{Azzollini-D'Avenia}[Proposition 2.2], any functions pair $(u,\phi)\in H_0^1(\Omega)\times H_0^1(\Omega)$ is a critical point of functional $J_q$ if and only if $u\in H_0^1(\Omega)$ is a critical point of functional $I_q$ with $\phi_u=\phi.$ Therefore, all solutions of problem \eqref{eq1.1} correspond to  critical points of the functional $I_q$ in the weak sense.

Next we collect some properties  of $\phi_u$, which will be used in this paper.
The following proposition can be proved by using similar arguments as in \cite{Ruiz}[Lemma 2.1].
\begin{Proposition}\label{lem1.1}
For any $u\in H_0^1(\Omega),$ the following hold:
\item{(i)} there exists $C>0$ such that $\|\phi_u\|\leq C\|u\|^q$ and $\int_{\Omega} |\nabla \phi_u|^2=\int_{\Omega} \phi_u |u|^q\leq C\|u\|^{2q};$
\item{(ii)} $\phi_u\geq 0$ and $\phi_{tu}=t^q \phi_u$ for any $t>0;$
\item{(iii)} if $u_n\rightharpoonup u$ in $H_0^1(\Omega),$ then
$\phi_{u_n}\rightharpoonup \phi_u$ in $H_0^1(\Omega)$ and $\lim\limits_{n\to +\infty}\int_{\Omega} \phi_{u_n} |u_n|^q=\int_{\Omega} \phi_u |u|^q.$
\end{Proposition}

\section{$p\in(2,6)$: nonlinear local perturbation}
In this section, we are devoted to proving Theorems \ref{prop} and \ref{theoremmain}. Let $p\in (2,6)$ and $\mu=1$. Since the argument of the proofs below are applicable for all $\lambda>0,$ in what follows, we also assume $\lambda=1.$

\subsection{Proof of Theorem \ref{prop}}
We shall use Nehari's method to prove the existence of  ground state solutions for problem \eqref{eq1.1}.
First we consider the  Nehari manifold
$$\mathcal{N}_q:=\left\{u\in H_0^1(\Omega)\backslash\{0\}: \int_{\Omega} |\nabla u|^2=\int_{\Omega} |u|^p +\int_{\Omega}\phi_u|u|^q\right\},$$
which is a natural constraint for the functional $I_q$.
Let $m_{q}$ be  the infimum of $I_q$ over Nehari manifold $\mathcal{N}_q$, that is,
$$
m_{q}:=\inf_{u\in \mathcal{N}_q} I_q(u).
$$

\begin{Lemma}\label{nehariproperty}
The following statements are true:
\item{(i)} there exists $\rho>0$ such that $\|u\|\geq \rho$ for all $u\in \mathcal{N}_q;$
\item{(ii)} $m_{q}>0;$
\item{(iii)} for each $v\in H_0^1(\Omega)\backslash\{0\},$ there exists a unique $s_v>0$ such that $s_v v\in \mathcal{N}_q$ and $$I_q(s_vv)=\sup\limits_{s\geq0}I_q(sv);$$
    Moreover, the map $H_0^1(\Omega)\backslash\{0\}\to (0,+\infty): v\mapsto s_v$ is continuous;
\item{(iv)}  $\inf\limits_{u\in\mathcal{N}_q} I_q(u)=\inf\limits_{v\in H_0^1(\Omega)\backslash\{0\}}\sup\limits_{s\geq0}I_q(sv).$
\end{Lemma}
\begin{proof}
For any $u\in \mathcal{N}_q,$ by Sobolev inequality and (i) of Proposition \ref{lem1.1}, we have
\begin{equation*}
\int_{\Omega} |\nabla u|^2=\int_{\Omega} |u|^p + \int_{\Omega} \phi_u |u|^q
\leq C_1 \|u\|^p+ C_2\|u\|^{2q}.
\end{equation*}
Then
\begin{equation}\label{eq3.1}
1\leq C_1 \|u\|^{p-2}+ C_2\|u\|^{2q-2},
\end{equation}
for some $C_1,C_2>0$ independent of $u.$ Since $p>2$ and $2q>2,$
there exists $\rho>0$ such that
$\|u\|\geq \rho$ for all $u\in \mathcal{N}_q.$ Otherwise, there exists a sequence $(u_n)_{n\geq 1}\subset \mathcal{N}_q$ such that $\|u_n\|\to 0$ as $n\to \infty.$ Then the right side of the inequality above tends to $0,$ which is impossible. Hence (i) follows.

Define $\theta:=\min\{p,2q\}>2.$ Since $I_q'(u)u=0$ for any $u\in \mathcal{N}_q$, by (ii) of Proposition \ref{lem1.1} and (i) , we have
\begin{equation*}\begin{array}{rll}
I_q(u)&=I_q(u)-\frac{1}{\theta}I_q'(u)u\\
&=(\frac{1}{2}-\frac{1}{\theta})\int_{\Omega} |\nabla u|^2+(\frac{1}{\theta}-\frac{1}{p})\int_{\Omega} |u|^p +(\frac{1}{\theta}-\frac{1}{2q}) \int_{\Omega} \phi_u |u|^q\\
&\geq (\frac{1}{2}-\frac{1}{\theta})\int_{\Omega} |\nabla u|^2 \\
&\geq (\frac{1}{2}-\frac{1}{\theta})\rho^2 >0.
\end{array}\end{equation*}
So (ii) follows.

For any $v\in H_0^1(\Omega)\backslash\{0\}$ and $s\geq0,$  we have
\begin{equation}\label{eq3.2}
\frac{d}{ds}I_q(sv)=I_q'(sv)v=s(\|v\|^2- s^{p-2}\int_{\Omega}|v|^p -s^{2q-2}\int_{\Omega}\phi_v |v|^q).
\end{equation}
Since $\frac{1}{s}I_q'(sv)v$ is positive for $s>0$ sufficiently small,  negative for $s$ sufficiently large, and is strictly decreasing with respect to $s$ in $(0,+\infty),$  there exists a unique $s_v>0$ such that $I_q'(s_vv)v=0$ and $s_vv\in \mathcal{N}_q.$
Moreover, the map $s\mapsto I_q(sv)$ is increasing for $0\leq s<s_v$ and decreasing for $s\geq s_v.$
Hence $I_q(s_vv)=\sup\limits_{s\geq0}I_q(sv).$

Now, we prove the continuity of the map $v\mapsto s_v.$ First, we assume  $v_n\to v$ in $H_0^1(\Omega)$ as $n\to \infty.$ In view of \eqref{eq3.2},  we have $I_q'(s_{v_n}v_n)v_n=0$
and $I_q'(s v_n)v_n\to -\infty$ uniformly in $n$ as $s\to+\infty.$ It is easy to verify  $(s_{v_n})_{n\geq 1}$ is bounded. Moreover, since $s_{v_n}v_n\in \mathcal{N}_q,$ it follows from (i) that $(s_{v_n})_{n\geq1}$ is bounded away from zero. Then there is a subsequence of $(s_{v_n})_{n\geq1}$ converging to some $s^*\in (0,+\infty).$ By \eqref{eq3.2}, we conclude that  $I_q'(s^*v)v=0$,  and the uniqueness of $s_v$ yields that $s^*=s_v.$  Therefore, $s_{v_n}\to s_v$ and (iii) follows.

In view of (iii), we have $I_q(s_vv)=\sup\limits_{s\geq0}I_q(sv)
\geq\inf\limits_{u\in \mathcal{N}_q}I_q(u).$ Thus
$$\inf\limits_{v\in H_0^1(\Omega)\backslash\{0\}}\sup\limits_{s\geq0}I_q(sv)\geq\inf\limits_{u\in \mathcal{N}_q}I_q(u).$$
On the other hand, for any $u\in\mathcal{N}_q,$
$$I_q(u)=\sup\limits_{s\geq0}I_q(su)\geq\inf\limits_{v\in H_0^1(\Omega)\backslash\{0\}}\sup\limits_{s\geq0}I_q(sv).$$
Therefore, (iv) follows immediately. We complete the proof.
\end{proof}

In spirit of \cite{Willem}[Chapter 4], we have the following lemma.
\begin{Lemma}\label{lemma3.2}
Define
$$m_q^1:=\inf_{u\in H_0^1(\Omega)\backslash\{0\}}\sup_{t\geq 0}I_q(tu),$$
$$m_q^2:=\inf_{\gamma\in\Gamma}\sup_{t\in[0,1]}I_q(\gamma(t)),$$
where
$$\Gamma=\{\gamma\in C([0,1],H_0^1(\Omega)): \gamma(0)=0, I_q(\gamma(1))<0\}.$$
Then
\begin{equation*}\label{mpstruc}
m_q=m_q^1=m_q^2.
\end{equation*}
\end{Lemma}
\begin{proof}
By (iv) of Lemma \ref{nehariproperty}, we have $m_q=m_q^1.$ For any $u\in H_0^1(\Omega)\backslash\{0\},$
there exists $t_u>0$ large enough such that $I_q(t_uu)<0$. Define $\gamma_u\in C([0,1],H_0^1(\Omega))$
by $$\gamma_u(s)=st_uu,\ s\in [0,1].$$
Then $\gamma_u\in \Gamma$ and $\max\limits_{s\in[0,1]}I_q(\gamma_u(s))\leq \sup\limits_{t\geq0}I_q(tu).$ Hence
$$m_q^2\leq\inf_{u\in H_0^1(\Omega)\backslash\{0\}}\sup_{t\geq 0}I_q(tu)=m_q^1.$$
On the other hand, set $\mathcal{B}_r=\{u\in H_0^1(\Omega):\|u\|\leq r, \ r>0 \}.$
It is easy to see that $I'_q(u)u\geq0$ for any $u\in \mathcal{B}_r$ with $r$ small enough. In addition,
for each $\gamma\in \Gamma,$ we have
$$I_q'(\gamma(1))\gamma(1)=2I_q(\gamma(1))+
(\frac{2}{p}-1)\int_\Omega |\gamma(1)|^p+(\frac{1}{q}-1)\int_\Omega\phi_{\gamma(1)} |\gamma(1)|^q<0 $$
since $ p\in (2,6)$ and $ q\in (1,5)$.
This implies that every $\gamma\in \Gamma$ must cross $\mathcal{N}_q$ and thus $m_q\leq m^2_q.$
We complete the proof.
\end{proof}

\noindent\textbf{Proof of Theorem \ref{prop}:}
Set $M=[0,1],$ $M_0=\{0,1\}$ and
$$\Gamma_0=\{\gamma_0:\{0,1\}\to H^1_0(\Omega)|\gamma_0(0)=0,\ I_q(\gamma_0(1))<0\}.$$
By Lemma \ref{lemma3.2}, we have
$$m_q^2:=\inf_{\gamma\in\Gamma}\sup_{t\in[0,1]}I_q(\gamma(t))=m_q>0=
\sup\limits_{\gamma_0\in \Gamma_0}\sup\limits_{t\in M_0}I_q(\gamma_0(t)).$$
 By the minimax principle (see \cite{Willem}[Theorem 2.8]), there exists a $(PS)_{m_{q}}$ sequence $(u_n)_{n\geq1}$ of $I_q$ such that
$$I_q(u_n)\to m_q,\ I_q'(u_n)\to 0\ \mbox{in} \ H^{-1}(\Omega)$$
as $n\to\infty.$ Then for $n$ large enough, we have
\begin{equation}\label{eq4.1}\begin{array}{rll}
m_q+1+\|u_n\| &\geq& I(u_n)-\frac{1}{\theta}I'(u_n)u_n\\
&\geq& (\frac{1}{2}-\frac{1}{\theta})\|u_n\|^2+(\frac{1}{\theta}-\frac{1}{p})\int_{\Omega} |u_n|^p +(\frac{1}{\theta}-\frac{1}{2q})\int_{\Omega} \phi_u |u|^q\\
&\geq& (\frac{1}{2}-\frac{1}{\theta})\|u_n\|^2.
\end{array}\end{equation}
Since $\theta:=\min\{p,2q\}>2,$  we deduce from \eqref{eq4.1} that $\|u_n\|$ is uniformly bounded. Up to a subsequence, there is $u^*\in H_0^1(\Omega)$ satisfying
\begin{equation*}
\begin{array}{rll}
u_n&\rightharpoonup u^*\ &\mbox{in}\ H_0^1(\Omega),\\
u_n&\to u^*\ &\mbox{in}\ L^s(\Omega)\ \mbox{ for any}\  s\in[1,6),\\
u_n(x)&\to u^*(x)\ &\mbox{a.e. in}\ \Omega.
\end{array}
\end{equation*}
This, combined with (iii) of Proposition \ref{lem1.1}, implies that $I_q'(u^*)=0.$
 Note that $I_q'(u_n)u_n\to 0$ as $n\to\infty$ and $I_q'(u^*)u^*=0.$ Then by (iii) of Proposition \ref{lem1.1}
 again, we deduce
 $$\int_{\Omega}|\nabla u_n|^2\to\int_{\Omega}|\nabla u^*|^2,\ n\to\infty.$$
 Since $u_n\rightharpoonup u^*$ in $H^1_0(\Omega),$ we have
$u_n\to u^*$ in $H_0^1(\Omega)$ and hence $\lim\limits_{n\to\infty}I_q(u_n)=I_q(u^*)=m_q.$ Therefore, $u^*$ is a ground state of problem \eqref{eq1.1}.

If $u$ is a ground state solution of  \eqref{eq1.1}, $|u|$ is also a ground state of \eqref{eq1.1}. By the elliptic regularity argument as in \cite{Moroz-Schaftingen}, we have $|u|\in C^2(\Omega).$ Applying the strong maximum principle to \eqref{eq1.1}, we have either $|u|>0$ or $|u|= 0.$ Since $u\neq 0,$ we conclude that either $u>0$ or $u<0.$ The proof is complete. \hfill{\QEDopen}

\subsection{Proof of Theorem \ref{theoremmain}}
In this section, our aim is to find least energy nodal solutions of \eqref{eq1.1}. First, we introduce a useful lemma which can be viewed as a variant of \cite{Lieb-Loss}[Theorem 9.8].
\begin{Lemma}\label{eq4.2}
For any $f,g\in C(\Omega),$ it holds
\begin{equation}\label{eql4.2}
\int_{\Omega} g \psi_f =\int_{\Omega} f \psi_g\leq \left(\int_{\Omega} f \psi_f \int_{\Omega} g \psi_g\right)^{1/2},
\end{equation}
where $\psi_f$ and $\psi_g$ denote the solutions of
\begin{equation}\label{eq4.3}
\left\{\begin{array}{lll}
-\Delta \psi_f&=f\ &\mbox{in}\ \Omega,\\
 \psi_f&=0\ &\mbox{on}\ \partial\Omega,
\end{array}\right.
 \end{equation}
 and
 \begin{equation}\label{eq4.4}
\left\{\begin{array}{lll}
-\Delta \psi_g&=g\ &\mbox{in}\ \Omega,\\
 \psi_g&=0\ &\mbox{on}\ \partial\Omega,
\end{array}\right.
 \end{equation}
respectively.
\end{Lemma}
\begin{proof}
Clearly, when  $g=0,$ \eqref{eql4.2} is valid. Without loss of generality, we assume $g\neq 0.$
Multiplying \eqref{eq4.3} and \eqref{eq4.4} by
$\psi_g$ and $\psi_f$, respectively,
 and integrating by parts, we have
\begin{equation}\label{eq4.5}\begin{array}{lll}
\int_{\Omega} g \psi_f=\int_{\Omega} \psi_f (-\Delta \psi_g)
=\int_{\Omega} (-\Delta\psi_f)  \psi_g
=\int_{\Omega} f\psi_g.
\end{array}\end{equation}
Moreover, multiplying \eqref{eq4.4} by
$\psi_g,$ we have
\begin{equation}\label{eq4.6}
\int_{\Omega} g \psi_g =\int_{\Omega} |\nabla \psi_g|^2\geq 0.
\end{equation}
Hence for any $\mu\in \mathbb{R},$ if we replace $g$ by $f-\mu g$ in \eqref{eq4.6}, we obtain
\begin{equation*}
\int_{\Omega} (f-\mu g)\psi_{f-\mu g}\geq 0.
\end{equation*}
Then by \eqref{eq4.5} and a direct computation,
\begin{equation}\label{eq4.7}
0\leq \int_{\Omega} (f-\mu g)\psi_{f-\mu g}=\int_{\Omega} f\psi_f-2\mu\int_{\Omega} g\psi_f+\mu^2\int_{\Omega} g\psi_g.
\end{equation}
Since $g\neq0,$ we have $\int_{\Omega} g \psi_g\neq0.$ Thus, taking  $\mu=\frac{\int_{\Omega} g \psi_f}{\int_{\Omega} g \psi_g}$  we deduce \eqref{eql4.2} from \eqref{eq4.7}.
\end{proof}
\begin{Remark}\label{remarko} Since $C_c^\infty(\Omega)$ is dense in $H_0^1(\Omega),$ by Lemma \ref{eq4.2}, for  any $u\in H_0^1(\Omega)$ we have
\begin{equation}\label{eq4.8}
E\leq \sqrt{E_1 E_2}\leq \frac{E_1+E_2}{2},
\end{equation}
where
\begin{equation}\label{eq4.9}
E_1:=\int_{\Omega} \phi_{u^+}|u^+|^q,\quad E_2:=\int_{\Omega} \phi_{u^-}|u^-|^q,
\quad E:=\int_{\Omega} \phi_{u^-}|u^+|^q=\int_{\Omega} \phi_{u^+}|u^-|^q.
\end{equation}
\end{Remark}

Now, we define the Nehari nodal set
\begin{equation*}
\mathcal{N}_{nod,q}:=\left\{u\in H_0^1(\Omega): I_q'(u)u^+=I_q'(u)u^-=0\ \mbox{and}\ u^{\pm}\neq 0\right\},
\end{equation*}
where  $u^+(x):=\max\{u(x),0\}\geq0$, $u^-(x):=\min\{u(x),0\}\leq0$ and
$$I_q'(u)u^{\pm}=\int_{\Omega} |\nabla u^{\pm}|^2-\int_{\Omega} |u^{\pm}|^p-\int_{\Omega} \phi_u |u^{\pm}|^q.$$
Let $m_{nod,q}$ be the infimum of $I_q$ over $\mathcal{N}_{nod,q},$ that is,
\begin{equation}\label{eq4.10}
m_{nod,q}:=\inf_{u\in \mathcal{N}_{nod,q}}I_q(u).
\end{equation}

\begin{Lemma}\label{lembddbelow} The following statements are true:
\item{(i)} $\|u\|\geq \rho$ for any $u\in \mathcal{N}_{nod,q},$ where $\rho$ is given in Lemma \ref{nehariproperty}.
\item{(ii)} $m_{nod,q}>0;$
\item{(iii)} Let $q\in(2,5)$ and $(u_n)_{n\geq1}\subset\mathcal{N}_{nod,q}$ be a bounded sequence of $I_q.$ Then
$$\liminf_{n\to \infty} \|u_n^{\pm}\|>0.$$
\end{Lemma}
\begin{proof}
Since $\mathcal{N}_{nod,q}\subset \mathcal{N}_q,$ (i) and (ii) follow from Lemma \ref{nehariproperty} immediately.

We finally prove  (iii).

By Sobolev inequality, (i) of Proposition \ref{lem1.1} and boundness of $(u_n)_{n\geq 1}$, we have
$$\|u_n^{\pm}\|^2=\int_{\Omega} |u_n^{\pm}|^p +\int_{\Omega} \phi_{u_n} |u_n^{\pm}|^q\leq C_1\|u_n^{\pm}\|^p+ C_2\|u_n\|^q\|u_n^{\pm}\|^q
\leq C_1\|u_n^{\pm}\|^p+ C_3\|u_n^{\pm}\|^q$$
for some constant $C_1, C_2,C_3>0$ independent of $n.$
Then
\begin{equation}\label{eq4.11}
1\leq C_1\|u_n^{\pm}\|^{p-2}+ C_3\|u_n^{\pm}\|^{q-2}.
\end{equation}
Since $p,q>2,$ it is easy to see that
$$\liminf_{n\to \infty} \|u_n^{\pm}\|>0.$$
Otherwise, there exists a subsequence $(u_{n_k})_{k\geq1}$ such that the right side of inequality \eqref{eq4.11} tends to 0 as $k\to \infty$, which is impossible. Therefore, (iii) holds.
\end{proof}

Moreover, we have the following lemma.
\begin{Lemma}\label{eq4.12}
Let $q\in (2,5).$ For each $u\in H_0^1(\Omega)$ with $u^{\pm}\neq 0,$ there exist $t_u,s_u>0$ such that
$$t_u u^+ +s_u u^-\in \mathcal{N}_{nod,q},$$
and the map $H^1_0(\Omega)\to (0,+\infty)^2:u\mapsto (t_u,s_u)$ is continuous.
\end{Lemma}
\begin{proof}
Let $u\in H_0^1(\Omega)$ with $u^{\pm}\neq 0$ be a fixed function. We define $J^u_1,J^u_2:[0,+\infty)^2 \to \mathbb{R}$ respectively by
\begin{equation*}
\begin{array}{lll}
J^u_1(t,s)&=&I_q'(tu^+ +su^-)tu^+\\
&=&t^2\|u^+\|^2-t^p\int_{\Omega}|u^+|^p-t^{2q}\int_{\Omega}\phi_{u^+}|u^+|^q-t^q s^q\int_{\Omega}\phi_{u^-}|u^+|^q
\end{array}
\end{equation*}
and
\begin{equation*}
\begin{array}{lll}
J^u_2(t,s)&=&I_q'(tu^+ +su^-)su^-\\
&=&s^2\|u^-\|^2-s^p\int_{\Omega}|u^-|^p-s^{2q}\int_{\Omega}\phi_{u^-}|u^-|^q-t^q s^q\int_{\Omega}\phi_{u^+}|u^-|^q.
\end{array}\end{equation*}
Let $R: (0,+\infty)\to \mathbb{R}^1$ be  $$R(r)=\left(\frac{r^{2-q}\|u^+\|^2}{4\int_{\Omega}\phi_{u^-}|u^+|^q}\right)^{1/q},$$
where $r\in(0,+\infty).$ Clearly, $R(r)\to +\infty$ as $r\to 0,$ because $q>2.$

Since $p\in(2,6)$ and $q\in(2,5),$ by Sobolev inequality and (i) of Proposition \ref{lem1.1}, there exists $r_1>0$ small enough such that for any $r\in(0,r_1),$
\begin{equation*}\begin{array}{lll}
r^2\|u^+\|^2-r^p\int_{\Omega}|u^+|^p-r^{2q}\int_{\Omega}\phi_{u^+}|u^+|^q&\geq
r^2\|u^+\|^2-C_1r^p\|u^+\|^p-r^{2q}C_2\|u^+\|^q\\
&\geq
\frac{1}{2}r^2\|u^+\|^2.
\end{array}\end{equation*}
Then for any $r\in(0,r_1),$
\begin{equation}\label{eql4.14}\begin{array}{lll}
J^u_1(r,R(r))&=&r^2\|u^+\|^2-r^p\int_{\Omega}|u^+|^p-r^{2q}\int_{\Omega}\phi_{u^+}|u^+|^q-r^q R^q(r)\int_{\Omega}\phi_{u^-}|u^+|^q\\
&\geq& \frac{1}{4}r^2\|u^+\|^2>0.
\end{array}\end{equation}
In addition, since $J^u_1(R(r),r)\to -\infty$ as $r\to 0,$ there exists $0<r_2\leq r_1$ such that for any $r\in(0,r_2),$
\begin{equation}\label{eql4.15}
J^u_1(R(r),r)=R^2(r)\|u^+\|^2-R^p(r)\int_{\Omega}|u^+|^p-R^{2q}(r)\int_{\Omega}\phi_{u^+}|u^+|^q-r^q R^q(r)\int_{\Omega}\phi_{u^-}|u^+|^q
<0.
\end{equation}
Since $\frac{\partial J^u_1}{\partial s}(t,s)<0$ in $(0,+\infty)^2,$ we infer from \eqref{eql4.14} and \eqref{eql4.15} that for any $r\in (0,r_2),$
\begin{equation}\label{eq4.14}
J^u_1(r,s)>0\ \mbox{and}\ J^u_1(R(r),s)<0\ \mbox{for any}\ s\in (r,R(r)).
\end{equation}
Similarly, there exists $\tilde{r}_2>0$ such that
for any $r\in (0,\tilde{r}_2),$
\begin{equation}\label{eq4.15}
J^u_2(t,r)>0\ \mbox{and}\ J^u_2(t,R(r))<0\ \mbox{for any}\ t\in (r,R(r)).
\end{equation}
Let $r_0=\frac{1}{2}\min\{r_2,\tilde{r}_2\}.$ Then \eqref{eq4.14} and \eqref{eq4.15} imply that
\begin{equation}\label{eq4.16}
J^u_1(r_0,s)>0\ \mbox{and}\ J^u_1(R(r_0),s)<0\ \mbox{for any}\ s\in (r_0,R(r_0))
\end{equation}
and
\begin{equation}\label{eq4.17}
J^u_2(t,r_0)>0\ \mbox{and}\ J^u_2(t,R(r_0))<0\ \mbox{for any}\ t\in (r_0,R(r_0)).
\end{equation}

Define the vector field
\begin{equation*}
V^u(t,s):=(J^u_1(t,s)/t, J^u_2(t,s)/s).
\end{equation*}
By applying   Miranda Theorem (see\cite{Miranda} or \cite{Vrahatis}) to $V^u(t,s)$ in $[r_0,R(r_0)]\times [r_0,R(r_0)]$, we conclude from \eqref{eq4.16} and \eqref{eq4.17} that  there exist $t_u, s_u\in (r_0,R(r_0))$ such that
$V^u(t_u,s_u)=(0,0).$ Hence, $t_u u^+ +s_u u^-\in \mathcal{N}_{nod,q}.$

To prove the continuity of the map $u\mapsto (t_u,s_u)$, now we assume  $u_n\to u$ in $H_0^1(\Omega)$ as $n\to\infty.$ First, we show that $(t_u,s_u)$ is a local strict maximizer of $F^u:(0,+\infty)^2\to \mathbb{R}$ defined by
$$F^u(t,s)=\frac{1}{2}\|tu^+ +su^-\|^2-\frac{1}{p}\int_{\Omega}|tu^+ +su^-|^p-\frac{1}{2q}\int_{\Omega}\phi_{tu^+ +su^-} |tu^+ +su^-|^q.$$
In fact, it is easy to see that
\begin{equation}\label{eq4.18}
\frac{\partial F^u}{\partial t}(t_u,s_u)=\frac{1}{t_u}J^u_1(t_u,s_u)=0, \quad \frac{\partial F^u}{\partial s}(t_u,s_u)=\frac{1}{s_u}J^u_2(t_u,s_u)=0.
\end{equation}
Since $p\in(2,6)$ and $q\in(2,5)$, we deduce
\begin{equation}\label{eq4.19}\begin{array}{rll}
&\frac{\partial^2 F^u}{\partial t^2}(t_u,s_u)\\
=&\|u^+\|^2-(p-1)t_u^{p-2}\int_{\Omega}|u^+|^p-(2q-1)t_u^{2q-2}\int_{\Omega}\phi_{u^+}|u^+|^q
-(q-1)t_u^{q-2}s_u^{q}\int_{\Omega}\phi_{u^-}|u^+|^q\\
=&(2-p)t_u^{p-2}\int_{\Omega} |u^+|^p +(2-2q)t_u^{2q-2}\int_{\Omega} \phi_{u^+}|u^+|^q +(2-q)t_u^{q-2}s_u^{q}\int_{\Omega} \phi_{u^-}|u^+|^q\\
=&(2-p)t_u^{p-2}B_{1}+(2-2q)t_u^{2q-2}E_{1}+(2-q)t_u^{q-2}s_u^{q}E\\
<&0,
\end{array}\end{equation}
\begin{equation}\label{eq4.20}
\frac{\partial^2 F^u}{\partial s^2}(t_u,s_u)
=(2-p)s_u^{p-2}B_{2}+(2-2q)s_u^{2q-2}E_{2}+(2-q)s_u^{q-2}t_u^{q}E,
\end{equation}
and
\begin{equation}\label{eq4.21}
\frac{\partial^2 F^u}{\partial t\partial s}(t_u,s_u)=\frac{\partial^2 F^u}{\partial s\partial t}(t_u,s_u)=-qt_u^{q-1}s_u^{q-1}E,
\end{equation}
where $E,E_1,E_2$ are defined in \eqref{eq4.9} and
\begin{equation*}
\begin{aligned}
&B_{1}=\int_{\Omega} |u^+|^p, \quad
B_{2}=\int_{\Omega} |u^-|^p.
\end{aligned}
\end{equation*}
By \eqref{eq4.8}, we have  $$t_u^{3q-2}s_u^{q-2}E_1+t_u^{q-2}s_u^{3q-2}E_2\geq 2t_u^{2q-2}s_u^{2q-2}\sqrt{E_1 E_2}\geq 2t_u^{2q-2}s_u^{2q-2}E.$$
This, combined with \eqref{eq4.19}-\eqref{eq4.21},  implies that
\begin{equation}\label{eq4.22}\begin{array}{lll}
&\det \operatorname{D}^2 F^u(t_u,s_u)\\
=&\det \left(\begin{array}{lll}
\frac{\partial^2 F^u}{\partial t^2} & \frac{\partial^2 F^u}{\partial t\partial s} \\
\frac{\partial^2 F^u}{\partial s\partial t} & \frac{\partial^2 F^u}{\partial s^2}
\end{array}\right)(t_u,s_u) \\
=&(2-p)^2(t_u s_u)^{2p-2}B_{1}B_{2}+(2-p)(2-2q)(t_u^{p-2}s_u^{2q-2}B_{1}E_{2}+t_u^{2q-2}s_u^{p-2}E_{1}B_{2})\\
&+(2-p)(2-q)(t_u^{p+q-2}s_u^{q-2}B_{1}+s_u^{p+q-2}t_u^{q-2}B_{2})E\\
&+(2-q)(2-2q)(t_u^{3q-2}s_u^{q-2}E_1+s_u^{3q-2}t_u^{q-2}E_2)E+[(2-q)^2-q^2]t_u^{2q-2}s_u^{2q-2}E^2\\
\geq &(2-p)^2(ts)^{2p-2}B_{1}B_{2}+(2-p)(2-2q)(t_u^{p-2}s_u^{2q-2}B_{1}E_{2}+t_u^{2q-2}s_u^{p-2}E_{1}B_{2})\\
&+(2-p)(2-q)(t_u^{p+q-2}s_u^{q-2}B_{1}+s_u^{p+q-2}t_u^{q-2}B_{2})E\\
&+8(q-1)(q-2)t_u^{2q-2}s_u^{2q-2}E^2\\
>&0,\quad \mbox{here we use the fact $p,q>2.$}
\end{array}\end{equation}
Therefore, it follows from \eqref{eq4.19}, \eqref{eq4.22} that $\operatorname{D}^2F^u(t_u,s_u)$ is a negative definite matrix. Thus by \eqref{eq4.18}, $(t_u,s_u)$ is a local strict maximum point of $F^u$ in $(0,+\infty)^2.$

Furthermore, by the local strict maximum property of $F^u,$ there exists $\epsilon_0>0$  such that for any $0<\epsilon<\epsilon_0,$ we have $(0,0)\notin DF^u(\partial \mathcal{B}_{\epsilon}(t_u,s_u))$ and
$$\deg(V^u,(0,0),\mathcal{B}_{\epsilon}(t_u,s_u))=\deg(DF^u,(0,0),\mathcal{B}_{\epsilon}(t_u,s_u))=sgn \det(\operatorname{D}^2 F^u)(t_u,s_u)=1,$$
where $\deg$ represents Brouwer degree,  and $\mathcal{B}_{\epsilon}(t_u,s_u)\subset \mathbb{R}^2$ denotes an open ball with radius $\epsilon$ centered at $(t_u,s_u)\subset \mathbb{R}^2.$

Note that $F^u(t,s)\to -\infty$ as $|(t,s)|\to \infty$ uniformly for $u$ in a bounded set. Then $|(t_{u_n},s_{u_n})|$ is bounded. In addition, it is easy to verify that $F^{u_n}\to F^u$ uniformly over any compact subset of $[0,+\infty)^2.$ Thus for $n$ large enough, $(0,0)\notin DF^{u_n}(\partial \mathcal{B}_{\epsilon})$ and by the properties of Brouwer degree, $$\deg(V^{u_n},(0,0),\mathcal{B}_{\epsilon}(t_u,s_u))=\deg(V^u,(0,0),\mathcal{B}_{\epsilon}(t_u,s_u))=1.$$ This implies that there exists $(t_{u_n},s_{u_n})\in B_{\epsilon}((t_u,s_u))$ such that $V^{u_n}(t_{u_n},s_{u_n})=(0,0).$
Since $\epsilon<\epsilon_0$ is arbitrary, by letting $\epsilon\to 0,$ we can conclude that $(t_{u_n},s_{u_n})\to (t_u,s_u)$ as $n\to \infty$. Hence the map $u\mapsto (t_u,s_u)$ is continuous.  We complete the proof.
\end{proof}

\begin{Remark}
Obviously, if the local term $|u|^{p-2}u$ disappears or appears with $p=2$ in \eqref{eq1.1}, one can infer from \eqref{eq4.22} that
$$\begin{array}{lll}&\det \operatorname{D}^2F^u(t_u,s_u)\\
&=(2-q)(2-2q)(t_u^{3q-2}s_u^{q-2}E_1+s_u^{3q-2}t_u^{q-2}E_2)E+[(2-q)^2-q^2]t_u^{2q-2}s_u^{2q-2}E^2\\
&\geq 8(q-1)(q-2)t_u^{2q-2}s_u^{2q-2}E^2>0\ \mbox{for}\ q>2\\
\end{array}$$
and
$$\det \operatorname{D}^2F^u(t_u,s_u)=[(2-q)^2-q^2]t_u^{2q-2}s_u^{2q-2}E^2=-4t_u^2s_u^2E^2<0\ \mbox{for}\ q=2.$$
However, when the local term $|u|^{p-2}u$ appears with $p>2$ in \eqref{eq1.1}, for  $q=2$  the sign of the following expressoin is not certain
$$\begin{array}{lll}
&\det \operatorname{D}^2F^u(t_u,s_u)\\
&=(2-p)^2(t_u s_u)^{2p-2}B_{1}B_{2}+(2-p)(2-2q)(t_u^{p-2}s_u^2B_{1}E_{2}+t_u^2s_u^{p-2}E_{1}B_{2})\\
&+[(2-q)^2-q^2]t_u^2s_u^2E^2\\
&=(2-p)^2(t_u s_u)^{2p-2}B_{1}B_{2}+2(p-2)(t_u^{p-2}s_u^2B_{1}E_{2}+t_u^2s_u^{p-2}E_{1}B_{2})-4t_u^2s_u^2E^2.\\
\end{array}$$
 This difference makes us consider term $|u|^{p-2}u$ and treat equation \eqref{eq1.1} by different methods in the proof of Theorem \ref{theoremmain} for $q>2$ and $q=2$. Furthermore, the proof for the existence of least energy nodal solutions to equation \eqref{eq1.1} with $p=2$ is different from the one with $p>2,$ which is stated in section 4.2.
\end{Remark}

For simplicity of notations, we shall write
\begin{equation}\label{eq4.23}
\begin{aligned}
 B_{n,1}&=\int_{\Omega} |u_n^+|^p, &
 & B_{n,2}&=\int_{\Omega} |u_n^-|^p.
\end{aligned}
\end{equation}
\begin{equation}\label{eq4.24}
E_{n,1}:=\int_{\Omega} \phi_{u_n^+}|u_n^+|^q,\quad E_{n,2}:=\int_{\Omega} \phi_{u_n^-}|u_n^-|^q,
\quad E_n:=\int_{\Omega} \phi_{u_n^-}|u_n^+|^q=\int_{\Omega} \phi_{u_n^+}|u_n^-|^q.
\end{equation}

\noindent\textbf{Proof of Theorem \ref{theoremmain}:}
We shall complete the proof by  distinguishing  three cases.

\noindent \textbf{Case 1.} \emph{Existence for $q\in (2,5).$}

By (ii) of Lemma \ref{lembddbelow} and Ekeland variational principle, we see that there exists a sequence $(u_n)_{n\geq1}\subset \mathcal{N}_{nod,q}$ such that
\begin{align}
I_q(u_n)&\leq m_{nod,q}+\frac{1}{n}, \\
I_q(u_n)&\leq I_q(v)+\frac{1}{n}\|u_n-v\|,\ \mbox{for any}\ v\in \mathcal{N}_{nod,q}.\label{eq4.25}
\end{align}
By using similar argument as in \eqref{eq4.1}, $(u_n)_{n\geq 1}$ are uniformly bounded in $H_0^1(\Omega).$ Then up to a subsequence, there exists $u\in H_0^1(\Omega)$ such that $u_n\rightharpoonup u$ in $H_0^1(\Omega).$ Moreover, we claim  $u\neq 0$. Since otherwise, it follows from $\|u_n\|^2=\int_{\Omega}|u_n|^p+\int_{\Omega}\phi_{u_n}|u_n|^q\to 0$ that $\|u_n\|\to 0$, which contradicts with Lemma \ref{lembddbelow}(iii).

In the following, we shall show that
\begin{equation}\label{eq4.26}
I_q'(u_n)\varphi\to 0\ \mbox{for any } \varphi\in H_0^1(\Omega)
\end{equation}
as $n\to +\infty$. First, for each $n\geq1$ and any fixed $\phi\in C_c^{\infty}(\Omega)$, we define two functions $\Phi_n, \Psi_n:\mathbb{R}\times \mathbb{R}_+^2\to \mathbb{R}$ by
\begin{align*}
\Phi_n(\delta,\tilde{t},\tilde{s})
=I_q'(\tilde{t}(u_n+\delta \phi)^+ +\tilde{s}(u_n+\delta \phi)^-)(\tilde{t}(u_n+\delta \phi)^+),\\
\Psi_n(\delta,\tilde{t},\tilde{s})
=I_q'(\tilde{t}(u_n+\delta \phi)^+ +\tilde{s}(u_n+\delta \phi)^-)(\tilde{s}(u_n+\delta \phi)^-).
\end{align*}
Obviously, $\Phi_n,\Psi_n,\frac{\partial \Phi_n}{\partial \tilde{t}},\frac{\partial \Phi_n}{\partial \tilde{s}},\frac{\partial \Psi_n}{\partial \tilde{t}},\frac{\partial \Psi_n}{\partial \tilde{s}}$ is continuous in some neighborhood $U\times V\subset \mathbb{R}\times \mathbb{R}^2_+$ of $(0,1,1).$
Since $u_n\in \mathcal{N}_{nod,q},$ we have
$$\Phi_n(0,1,1)=\Psi_n(0,1,1)=0$$
and in view of \eqref{eq4.19}, \eqref{eq4.20} and \eqref{eq4.21},
\begin{equation}\label{eq4.27}\begin{array}{lll}
\frac{\partial \Phi_n}{\partial \tilde{t}}(0,1,1)&=2\| u_n^+\|^2-p\int_{\Omega} |u_n^+|^p-2q\int_{\Omega} \phi_{u_n^+}|u_n^+|^q -q\int_{\Omega} \phi_{u_n^-}|u_n^+|^q\\
&=(2-p)B_{n,1}+(2-2q)E_{n,1}+(2-q)E_n,
\end{array}\end{equation}
\begin{equation}\begin{array}{lll}
\frac{\partial \Psi_n}{\partial \tilde{s}}(0,1,1)
&=(2-p)B_{n,2}+(2-2q)E_{n,2}+(2-q)E_n,
\end{array}\end{equation}
and
\begin{equation}\label{eq4.28}
\frac{\partial \Phi_n}{\partial \tilde{s}}(0,1,1)=\frac{\partial \Psi_n}{\partial \tilde{t}}(0,1,1)=-q\int_{\Omega} \phi_{u_n^-}|u_n^+|^q=-qE_n.
\end{equation}
Here $B_{n,i},E_{n,i},E_n, i=1,2$ are defined as in \eqref{eq4.23} and \eqref{eq4.24}.
Since $p\in(2,6), q\in (2,5)$, by \eqref{eq4.22}, we infer from \eqref{eq4.8}, \eqref{eq4.27}-\eqref{eq4.28} that
\begin{equation}\label{eq4.29}\begin{array}{lll}
&\det \left(\begin{array}{lll}
\frac{\partial \Phi_n}{\partial \tilde{t}} & \frac{\partial \Phi_n}{\partial \tilde{s}} \\
\frac{\partial \Psi_n}{\partial \tilde{t}} & \frac{\partial \Psi_n}{\partial \tilde{s}}
\end{array}\right)(0,1,1) \\
\geq &(2-p)^2B_{n,1}B_{n,2}+(p-2)(2q-2)(B_{n,1}E_{n,2}+E_{n,1}B_{n,2})\\
&+(p-2)(q-2)(B_{n,1}+B_{n,2})E_n+8(q-1)(q-2)E_n^2\\
>&0.
\end{array}\end{equation}
Then the implicit function theorem yields that there exist $\delta_n>0$ and two functions $\tilde{t}_n(\delta),\tilde{s}_n(\delta)\in C^0((-\delta_n,\delta_n),\mathbb{R})$ such that
$$\tilde{t}_n(0)=\tilde{s}_n(0)=1$$ and
\begin{equation}\label{eq4.30}
\Phi_n(\delta,\tilde{t}_n(\delta),\tilde{s}_n(\delta))=\Psi_n(\delta,\tilde{t}_n(\delta),\tilde{s}_n(\delta))=0\ \mbox{for all} \ \delta\in(-\delta_n,\delta_n).
\end{equation}
This implies that
\begin{equation*}
\tilde{t}_n(\delta)(u_n+\delta \phi)^+ +\tilde{s}_n(\delta)(u_n+\delta \phi)^-\in N_{nod,q}
\end{equation*}
for all $\delta \in (-\delta_n,\delta_n).$

Next, if we write $u_n+\delta \phi$ by $u_{n,\delta}$ for simplicity and replace $v$ by $\tilde{t}_n(\delta)u_{n,\delta}^+ +\tilde{s}_n(\delta)u_{n,\delta}^-$ in \eqref{eq4.25} and use the Taylor expansion, i.e.,
\begin{equation*}\begin{array}{lll}
&&I_q(\tilde{t}_n(\delta)u_{n,\delta}^+ +\tilde{s}_n(\delta)u_{n,\delta}^-)\\
&=&I_q(u_n)+I_q'(u_n)(\tilde{t}_n(\delta)u_{n,\delta}^+ +\tilde{s}_n(\delta)u_{n,\delta}^- -u_n) +o(\|\tilde{t}_n(\delta)u_{n,\delta}^+ +\tilde{s}_n(\delta)u_{n,\delta}^--u_n\|)\\
&=&I_q(u_n)+I_q'(u_n)((\tilde{t}_n(\delta)-1)u_{n,\delta}^+ +(\tilde{s}_n(\delta)-1)u_{n,\delta}^- +\delta\phi)\\
&&\quad +o(\|(\tilde{t}_n(\delta)-1)u_{n,\delta}^+ +(\tilde{s}_n(\delta)-1)u_{n,\delta}^- +\delta\phi\|)\\
&=&I_q(u_n)+\delta I_q'(u_n)\phi+(\tilde{t}_n(\delta)-1)I_q'(u_n)u_{n,\delta}^+ +(\tilde{s}_n(\delta)-1)I_q'(u_n)u_{n,\delta}^-\\
&&\quad +o(\|(\tilde{t}_n(\delta)-1)u_{n,\delta}^+ +(\tilde{s}_n(\delta)-1)u_{n,\delta}^- +\delta\phi\|)\\
&=&I_q(u_n)+\delta I_q'(u_n)\phi +(\tilde{t}_n(\delta)-1)I_q'(u_n)[u_{n,\delta}^+-u_n^+] +(\tilde{s}_n(\delta)-1)I_q'(u_n)[u_{n,\delta}^--u_n^-]\\
&&\quad +o(\|(\tilde{t}_n(\delta)-1)u_{n,\delta}^+ +(\tilde{s}_n(\delta)-1)u_{n,\delta}^- +\delta\phi\|),
\end{array}\end{equation*}
then we have
\begin{equation}\label{eq4.34}\begin{array}{lll}
I_q'(u_n)\phi&\geq -\frac{1}{n\delta}\|(\tilde{t}_n(\delta)-1)u_{n,\delta}^+ +(\tilde{s}_n(\delta)-1)u_{n,\delta}^- +\delta\phi\|\\ &-\frac{1}{\delta}\left\{(\tilde{t}_n(\delta)-1)I_q'(u_n)[u_{n,\delta}^+-u_n^+] +(\tilde{s}_n(\delta)-1)I_q'(u_n)[u_{n,\delta}^--u_n^-]\right\}\\
&+o(\|(\tilde{t}_n(\delta)-1)u_{n,\delta}^+ +(\tilde{s}_n(\delta)-1)u_{n,\delta}^- +\delta\phi\|/\delta).
\end{array}\end{equation}

We claim that $\frac{\tilde{t}(\delta)-1}{\delta}$ and  $\frac{\tilde{s}(\delta)-1}{\delta}$ is bounded for $\delta$ near $0.$ In fact, without loss of generality, suppose on the contrary that there exists $\delta_j\to 0$ such that $a_{\infty}:=\lim\limits_{j\to\infty}\frac{\tilde{t}(\delta_j)-1}{\delta_j}\to +\infty$ as $j\to+\infty.$ If it is necessary, we may still denote by $b_{\infty}:=\lim\limits_{j\to\infty}\frac{\tilde{s}(\delta_j)-1}{\delta_j}$, where the limit may be up to a subsequence.

Note that the facts of $I_q'(u_n)u_n^+$ and $I_q'(u_{n,\delta})u_{n,\delta}^+$ imply
\begin{equation}\begin{array}{lll}\label{kedao}
0&=&(\tilde{t}^2(\delta_j)-1)\|u_{n,\delta_j}^+\|^2
-(\tilde{t}^p(\delta_j)-1)\|u_{n,\delta_j}^+\|_{L^p}^p
-(\tilde{t}^{2q}(\delta_j)-1)\int_{\Omega}\phi_{u_{n,\delta_j}^+}|u_{n,\delta_j}^+|^{q}\\
&&-\tilde{s}^q(\delta_j)(\tilde{t}^q(\delta_j)-1)\int_{\Omega}\phi_{u_{n,\delta_j}^-}|u_{n,\delta_j}^+|^{q}
-(\tilde{s}^q(\delta_j)-1)\int_{\Omega}\phi_{u_{n,\delta_j}^-}|u_{n,\delta_j}^+|^{q}\\
&&+\|u_{n,\delta_j}^+\|^2
-\|u_{n,\delta_j}^+\|_{L^p}^p
-\int_{\Omega}\phi_{u_{n,\delta_j}^+}|u_{n,\delta_j}^+|^{q}
-\int_{\Omega}\phi_{u_{n,\delta_j}^-}|u_{n,\delta_j}^+|^{q}\\
&&-[\|u_n^+\|^2
-\|u_n^+\|_{L^p}^p
-\int_{\Omega}\phi_{u_n^+}|u_n^+|^{q}
-\int_{\Omega}\phi_{u_n^-}|u_n^+|^{q}].
\end{array}\end{equation}
Since  $\tilde{t}(\delta_j)\to 1$ and $\tilde{s}(\delta_j)\to 1$, we have
\begin{equation}\label{tquyu}\begin{array}{lll}
\frac{\tilde{t}^2(\delta_j)-1}{\delta_j}=\frac{\tilde{t}(\delta_j)-1}{\delta_j}\frac{\tilde{t}^2(\delta_j)-1}{\tilde{t}(\delta_j)-1}\to 2a_{\infty},
\frac{\tilde{t}^p(\delta_j)-1}{\delta_j}\to pa_{\infty},
\frac{\tilde{t}^{2q}(\delta_j)-1}{\delta_j}\to 2q a_{\infty},
\frac{\tilde{s}^{q}(\delta_j)-1}{\delta_j}\to q b_{\infty},
\end{array}\end{equation}
Moreover, since
$$\begin{array}{lll}
2[u_{n,\delta_j}^+ -u_n^+]&=|u_{n,\delta_j}|+u_{n,\delta_j}-(|u_n|+u_n)=|u_{n,\delta_j}|-|u_n|+\delta_j \phi
=\frac{2\delta_j u_n\phi+\delta_j^2 \phi^2}{|u_{n,\delta_j}|+|u_n|}+\delta_j\phi,
\\
2[u_{n,\delta_j}^- -u_n^-]&=u_{n,\delta_j}-|u_{n,\delta_j}|-(u_n-|u_n|)=|u_n|-|u_{n,\delta_j}|+\delta_j \phi
=\frac{-2\delta_j u_n\phi-\delta_j^2 \phi^2}{|u_{n,\delta_j}|+|u_n|}+\delta_j\phi
\end{array}$$
we have
$$
2\|[u_{n,\delta_j}^+ -u_n^+]/\delta_j\|=\|\frac{2u_n\phi+\delta_j \phi^2}{|u_{n,\delta_j}|+|u_n|}+\phi\|\to \|\frac{u_n\phi}{|u_n|}+\phi\|,$$
$$2\|[u_{n,\delta_j}^- -u_n^-]/\delta_j\|=\|\frac{-2u_n\phi-\delta_j \phi^2}{|u_{n,\delta_j}|+|u_n|}+\phi\|\to \|\frac{-u_n\phi}{|u_n|}+\phi\|$$
and
$$\frac{1}{\delta_j}\left|\|u_{n,\delta_j}^+\|^2-\|u_n^+\|^2\right|\leq \frac{1}{\delta_j}\left|\|u_{n,\delta_j}^+-u_n^+\|^2\right|\to \frac{1}{2}\|\frac{u_n\phi}{|u_n|}+\phi\|,$$
$$\frac{1}{\delta_j}\left|\|u_{n,\delta_j}^-\|^2-\|u_n^-\|^2\right|\leq \frac{1}{\delta_j}\left|\|u_{n,\delta_j}^--u_n^-\|^2\right|\to \frac{1}{2}\|\frac{-u_n\phi}{|u_n|}+\phi\|$$
 as $\delta_j\to 0.$ This shows that  $\frac{1}{\delta_j}\left|\|u_{n,\delta_j}^+\|^2-\|u_n^+\|^2\right|$ is bounded for small $\delta_j$. Furthermore, using similar arguments, we can prove that
\begin{equation}\begin{array}{lll}
&\frac{1}{\delta_j}\left\{\|u_{n,\delta_j}^+\|^2
-\|u_{n,\delta_j}^+\|_{L^p}^p
-\int_{\Omega}\phi_{u_{n,\delta_j}^+}|u_{n,\delta_j}^+|^{q}
-\int_{\Omega}\phi_{u_{n,\delta_j}^-}|u_{n,\delta_j}^+|^{q}-(\|u_n^+\|^2\right.\\
&\left.
-\|u_n^+\|_{L^p}^p
-\int_{\Omega}\phi_{u_n^+}|u_n^+|^{q}
-\int_{\Omega}\phi_{u_n^-}|u_n^+|^{q})\right\}
\end{array}\end{equation}
is bounded for all small $\delta_j.$
By letting $\delta_j\to 0,$ this together with \eqref{kedao},\eqref{tquyu}, gives that
$$\begin{array}{lll}
a_{\infty}(2\|u_n^+\|^2
-p\|u_n^+\|_{L^p}^p
-2q\int_{\Omega}\phi_{u_n^+}|u_n^+|^{q}-q \int_{\Omega}\phi_{u_n^-}|u_n^+|^{q})
-q b_{\infty}\int_{\Omega}\phi_{u_n^-}|u_n^+|^{q}=C_{1,n}
\end{array}$$
for some $C_{1,n}\in \mathbb{R}.$ That is,
\begin{equation}\label{gx1}
a_{\infty}A_n^1- b_{\infty}B_n= C_{1,n},
\end{equation}
where
$$\begin{array}{lll}
A_n^1:=
(2-p)\|u_n^+\|_{L^p}^p+
(2-2q)\int_{\Omega}\phi_{u_n^+}|u_n^+|^{q}+(2-q) \int_{\Omega}\phi_{u_n^-}|u_n^+|^{q}<0,\\
B_n:=(2-q) b_{\infty}\int_{\Omega}\phi_{u_n^-}|u_n^+|^{q}.
\end{array}$$
Then it follows from the assumption $a_{\infty}=+\infty$ that $b_{\infty}=-\infty.$
Moreover, similar arguments lead to that
\begin{equation}\label{guanxi2}
b_{\infty}A_n^2- a_{\infty}B_n= C_{2,n}
\end{equation}
for some $C_{2,n}\in \mathbb{R},$ where
$$A_n^2:=
(2-p)\|u_n^-\|_{L^p}^p+
(2-2q)\int_{\Omega}\phi_{u_n^-}|u_n^-|^{q}+(2-q) \int_{\Omega}\phi_{u_n^-}|u_n^+|^{q}.
$$
Note that $C_{1,n},C_{2,n}$ are bounded due to the boundness of $u_n.$  Thus, for large $n,$ it follows from
\eqref{gx1} and \eqref{guanxi2} that
\begin{equation}\label{dayu}
a_{\infty}(A_n^1A_n^2-B_n^2)=C_{1,n}A_n^2+C_{2,n}B_n.
\end{equation}
But since \eqref{eq4.27}-\eqref{eq4.29} give that
$$A_n^1A_n^2-B_n^2=\det \left(\begin{array}{lll}
\frac{\partial \Phi_n}{\partial \tilde{t}} & \frac{\partial \Phi_n}{\partial \tilde{s}} \\
\frac{\partial \Psi_n}{\partial \tilde{t}} & \frac{\partial \Psi_n}{\partial \tilde{s}}
\end{array}\right)(0,1,1)>0,$$
it yields a contradiction in \eqref{dayu}. So the claim holds true.

By using this claim and letting $\delta\to 0$, one can deduce from \eqref{eq4.34} that
$$|I_q'(u_n)\phi|\leq \frac{C}{n}$$
for some $C>0$ independent of $n.$
Furthermore, the arbitrary choice of $\phi\in C_c^{\infty}(\Omega)$ yields that for any $\varphi\in H_0^1(\Omega),$
$$I_q'(u_n)\varphi\to 0\ \mbox{as}\ n\to \infty.$$
Thus \eqref{eq4.26} follows.

Since $u_n\rightharpoonup u$ in $H_0^1(\Omega),$
$u_n\to u$ in $L^s(\Omega)$ with $ s\in[1,6).$ This, together with (iii) of Proposition \ref{lem1.1} and \eqref{eq4.26}, implies that $I_q'(u)\varphi=0$ for any $\varphi\in H_0^1(\Omega).$ Take $\varphi=u^{\pm}.$  Then there holds
\begin{equation}\label{eq4.36}
\|u^{\pm}\|^2=\int_{\Omega} |u^{\pm}|^p+\int_{\Omega} \phi_u |u^{\pm}|^q.
\end{equation}
In addition, we have
\begin{equation}\label{eq4.37}
\int_{\Omega} |u_n^{\pm}|^p\to \int_{\Omega} |u^{\pm}|^p\ \mbox{and}\ \int_{\Omega} \phi_{u_n} |u_n^{\pm}|^q\to \int_{\Omega} \phi_{u} |u^{\pm}|^q,
\end{equation}
Since $u_n\in \mathcal{N}_{nod,q},$ we infer from \eqref{eq4.36} and \eqref{eq4.37} that $u_n^{\pm}\to u^{\pm}$ in $H_0^1(\Omega).$ Thus, it follows from (iii) of Lemma \ref{lembddbelow} that $u^{\pm}\neq 0.$ Moreover,
$$\lim\limits_{n\to \infty}I_q(u_n)=I_q(u)=m_{nod,q}.$$
Therefore, $u$ is a least energy nodal solution of \eqref{eq1.1} for $q\in (2,5)$.

\noindent\textbf{Case 2.} \emph{Existence for $q=2.$}

First, we shall show
\begin{equation}\label{eq4.38}
\limsup_{q\searrow 2} m_{nod,q}\leq m_{nod,2}
\end{equation}
where $m_{nod,2}$ is defined in \eqref{eq4.10}.

Let $w\in \mathcal{N}_{nod,2}$.  For any $q\in[2,5)$, we define
$i_q:[0,\infty)\times [0,\infty)\to \mathbb{R}$ by
\begin{equation*}\begin{array}{lll}
i_q(t,s)&=&\frac{t^2}{2}\|w^+\|^2+\frac{s^2}{2}\|w^-\|^2
-\frac{t^p}{p}\int_{\Omega}|w^+|^p-\frac{s^p}{p}\int_{\Omega}|w^-|^p
-\frac{t^{2q}}{2q}\int_{\Omega} \phi_{w^+} |w^+|^q\\
&&-\frac{s^{2q}}{2q}\int_{\Omega} \phi_{w^-} |w^-|^q-\frac{t^q s^q}{q}\int_{\Omega} \phi_{w^-} |w^+|^q.
\end{array}\end{equation*}
Clearly, $\frac{\partial i_2}{\partial t}(1,1)=\frac{\partial i_2}{\partial s}(1,1)=0.$ Since $q\mapsto \int_{\Omega}\phi_w |w^{\pm}|^q$ is continuous, we have
$i_q\to i_2$ as $q\to 2$ uniformly on every compact set of $[0,+\infty)\times [0,+\infty).$ Then by similar arguments as in the proof of Lemma \ref{eq4.12}, there exists $t_q, s_q\in (0,+\infty)$ such that
$t_q w^+ +s_q w^-\in \mathcal{N}_{nod,q}$ and $(t_q, s_q)\to (1,1)$ as $q\to 2.$
This implies that
\begin{equation*}
\lim\limits_{q\searrow 2}I_q(t_q w^+ +s_q w^-)=I_2(w).
\end{equation*}
Note that $m_{nod,q}\leq I_q(t_q w^+ +s_q w^-).$
 Hence, by the  arbitrary choice of $w\in \mathcal{N}_{nod,2},$ \eqref{eq4.38} follows immediately.

Second, according to Case 1, a least energy nodal solution  $u_q\in H_0^1(\Omega)$  of \eqref{eq1.1} with $q\in(2,5)$ exists. Since $m_{nod,2}\leq I_2(w)<+\infty$,  it follows from \eqref{eq4.38} that if $q$ is close to 2, $\|u_q\|$ are uniformly bounded. Then there exists $u\in H_0^1(\Omega)$ such that, up to a subsequence, $u_q\rightharpoonup u$ in $H_0^1(\Omega)$ as $q\searrow 2.$ Notice that  $I_q'(u_q)\phi=0$ for any $\phi\in H_0^1(\Omega)$. Then we have
\begin{equation*}
I_{2}'(u)\phi=0\ \mbox{for any $\phi\in H_0^1(\Omega)$}.
\end{equation*}
Taking $\phi=u$ and using the facts that
$$I_q'(u_q)u_q=0,\ \int_{\Omega} |u_q|^p\to \int_{\Omega} |u|^p
\mbox{\, and \,}\ \int_{\Omega} \phi_{u_q}|u_q|^{q}\to \int_{\Omega} \phi_{u} |u|^2,$$
we conclude  that
\begin{equation}\label{eq4.39}
\int_{\Omega}|\nabla u_q|^2\to\int_{\Omega}|\nabla u|^2 \ \mbox{as}\  q\to 2.
\end{equation}
Then  by \eqref{eq3.1} and \eqref{eq4.39},  we obtain
\begin{equation*}
1\leq C_1 \|u\|^{p-2}+ C_2\|u\|^{2}.
\end{equation*}
Hence $u\neq 0.$

Now we shall prove $u^{\pm}\neq 0.$  Indeed, by \eqref{eq4.39}, we have
\begin{equation}\label{eq4.40}
\lim \limits_{q\searrow 2}\|u_q^{\pm}\|=\|u^{\pm}\|.
\end{equation}
Then it suffices to show
\begin{equation}\label{eq4.41}
\lim_{q\searrow 2}\|u_q^{\pm}\|^2>0.
\end{equation}
We argue  by contradiction. Without loss of generality, suppose on the contrary that there is a sequence $(q_n)_{n\geq 1}\subset (2,3)$ such that
$q_n\searrow2$ as $n\to\infty$, and
\begin{equation}\label{eq4.42}
\lim\limits_{n\to\infty}\|u_{q_n}^-\|=0.
\end{equation}
Let $v_{q_n}=\frac{u_{q_n}^-}{\|u_{q_n}^-\|}.$ Then up to a subsequence, $v_{q_n}\rightharpoonup v$ in $H_0^1(\Omega)$ for some $v\in H_0^1(\Omega)$ as $n\to\infty.$
In addition,   $u_{q_n}\in \mathcal{N}_{nod,q_n}$ implies that
\begin{equation}\label{eq4.43}
1=\|u_{q_n}^-\|^{p-2}\int_{\Omega}|v_{q_n}|^p
+\|u_{q_n}^-\|^{q_n-2}\int_{\Omega}\phi_{u_{q_n}}|v_{q_n}|^{q_n}.
\end{equation}
Since $p,q_n>2,$ we infer from \eqref{eq4.42}  that for $n$ large enough,
$$\|u_{q_n}^-\|^{p-2}\int_{\Omega}|v_{q_n}|^p\leq \frac{1}{2}\quad
\mbox{and}\quad
\|u_{q_n}^-\|^{q_n-2}\leq 1.$$
Then by \eqref{eq4.43}, we have
\begin{equation}\label{eq4.44}
\int_{\Omega}\phi_{u_{q_n}}|v_{q_n}|^{q_n}\geq \frac{1}{2}\|u_{q_n}^-\|^{2-q_n}\geq \frac{1}{2}.
\end{equation}
On the other hand, by (i) of Proposition \ref{lem1.1} and  the uniform boundness of $\|u_{q_n}\|$, we obtain
\begin{equation}\label{eq4.45}
\int_{\Omega}\phi_{u_{q_n}} |v_{q_n}|^{q_n}\leq C_1 \|u_{q_n}\|^{q_n} \left(\int_{\Omega}|v_{q_n}|^{\frac{6q_n}{5}}\right)^{\frac{5}{6}}\leq C\left(\int_{\Omega}|v_{q_n}|^{\frac{6q_n}{5}}\right)^{\frac{5}{6}}
\end{equation}
for some $C>0$ independent of $n.$ Then we can deduce from \eqref{eq4.44} and \eqref{eq4.45} that
\begin{equation}\label{eq4.46}
\liminf\limits_{n\to \infty}\int_{\Omega}|v_{q_n}|^{\frac{6q_n}{5}}\geq(\frac{1}{2C})^{\frac{6}{5}}.
\end{equation}
Since $q_n\in (2,3)$, for each $n$, we have
$$\int_{\Omega}|v_{q_n}-v|^{\frac{6q_n}{5}}\leq\int_{\Omega}|v_{q_n}-v|^{\frac{12}{5}}
+\int_{\Omega}|v_{q_n}-v|^{\frac{18}{5}}.$$
This, combined with Rellich-Kondrachov compactness theorem, implies
$$\begin{array}{lll}
\big{|}\int_{\Omega}|v_{q_n}|^{\frac{6q_n}{5}}-\int_{\Omega}|v|^{\frac{12}{5}}\big{|}
&\leq& \int_{\Omega}|v_{q_n}-v|^{\frac{6q_n}{5}}
+\int_{\Omega}(|v|^{\frac{6q_n}{5}}-|v|^{\frac{12}{5}})\\
&\leq& \int_{\Omega}|v_{q_n}-v|^{\frac{12}{5}}+
\int_{\Omega}|v_{q_n}-v|^{\frac{18}{5}}+\int_{\Omega}(|v|^{\frac{6q_n}{5}}-|v|^{\frac{12}{5}})\\
&\to &0\quad \mbox{as $n\to\infty.$}
\end{array}$$
Thus we infer from \eqref{eq4.46} that
\begin{equation}\label{eq4.47}
\int_{\Omega}|v(x)|^{\frac{12}{5}}dx>0.
\end{equation}
Note that
$\{x\in \Omega: v_{q_n}(x)<0\}=\{x\in \Omega: u_{q_n}(x)<0\}$.
Since $v_{q_n}(x)\to v(x)$ and $u_{q_n}(x)\to u(x)$ a.e. in $\Omega$ as $n\to \infty,$ we have
$$\{x\in \Omega: v(x)<0\}\subset \{x\in \Omega: u(x)\leq 0\}.$$
Then it follows from \eqref{eq4.47} that
\begin{equation*}
meas\{x\in \Omega: u(x)\leq 0\}\neq 0,
\end{equation*}
where $meas$ denotes the Lebesgue measure. This, combined with \eqref{eq4.40} and \eqref{eq4.42},  implies
\begin{equation}\label{eq4.48}
meas\{x\in \Omega: u(x)=0\}\neq 0.
\end{equation}
Moreover, since $(u,\phi_u)$ is a weak solution of \eqref{eq1.1}, by the elliptic regularity argument, we have $u\in C^2(\bar{\Omega}).$ Then,  by strong maximum principle, we derive that either $u\equiv 0$ or $u>0.$ Therefore it follows from \eqref{eq4.48} that $u\equiv 0,$ which contradicts  with the fact that $u\neq 0.$
Thus $u^-\neq 0.$ Similarly, we can conclude $u^+\neq 0.$ Therefore, \eqref{eq4.41} follows.

Finally, by the arguments above, we have shown that  $u$ is a nontrivial critical point of $I_2$ and $u^{\pm}\neq 0.$  Then $u\in \mathcal{N}_{nod,2}$. This combined with \eqref{eq4.38} and \eqref{eq4.39}, implies
$$m_{nod,2}\leq I_2(u)=\lim\limits_{q\to 2}I_q(u_q)\leq m_{nod,2}.$$ Hence,
$u$ is a least energy nodal solution of \eqref{eq1.1} for $q=2.$

\noindent\textbf{Case 3.} \emph{Nonexistence for $q\in(1,2).$}

Note that for any $u\in \mathcal{N}_q$,  we have  $u_q\in H_0^1(\Omega)$ and  $I_q(u)=I_q(|u|)$.   Then
$$m_{q}=\inf \left\{I_q(u):  u\in \mathcal{N}_q, u\geq 0\ \mbox{a.e. in}\ \Omega\right\}.$$
For any $u\in \mathcal{N}_q$ with $u\geq 0,$  we can find a sequence $(u_n)_{n\geq1}\subset C_c^{\infty}(\Omega)$ satisfying  $u_n\geq 0$ and $u_n\to u$ in $H^1_0(\Omega)$ as $n\to\infty.$
By (iii) of Lemma \ref{nehariproperty}, for each $n$, there exists a unique  $s_n>0$ such that $s_n u_n\in \mathcal{N}_q$ and $s_n\to 1$ as $n\to\infty.$ So there holds
\begin{equation*}
m_{q}=\inf \left\{I_q(u): u\in C_c^{\infty}(\Omega)\cap \mathcal{N}_q, u\geq 0\ \mbox{in}\ \Omega\right\}.
\end{equation*}

Let us take a function $u\in C_c^{\infty}(\Omega)\cap \mathcal{N}_q$ with $u\geq 0$ in $\Omega.$ Without loss of generality, we assume that
$$\mbox{dist}(supp\ u,\partial \Omega)=2\tilde{\delta}$$
for some $\tilde{\delta}>0.$ Let $a\in \Omega$ be such that $B_{\tilde{\delta}/2}(a)\subset \Omega\backslash supp\ u.$
We define a family of functions $u_{\delta}:\Omega\to \mathbb{R}$ with $\delta \in (0,\tilde{\delta}/2)$ by
$$u_{\delta}(x)=u(x)-\delta^{\frac{2}{2-q}}\eta\big{(}\frac{x-a}{\delta}\big{)}\quad \mbox{for any } x\in \Omega,$$
where $\eta\in C_c^{\infty}(\mathbb{R}^3,\mathbb{R})$ is a cut-off function such that $\eta(x)=1$ in $B_1(0)$ and $\eta(x)=0$ in $\mathbb{R}^3\backslash B_2(0)$ with $0\leq \eta\leq 1.$
Obviously, $u_{\delta}^+=u,$ $u_{\delta}^-=-\delta^{\frac{2}{2-q}}\eta(\frac{\cdot-a}{\delta})\in H_0^1(\Omega)$ and $u_{\delta} \to u$ in $H_0^1(\Omega)$ as $\delta\to 0$ due to the fact that $q\in (1,2).$
Furthermore, by direct calculation, we have
\begin{equation}\label{eq4.49}
\begin{array}{lll}
\int_{\Omega} \phi_{u_{\delta}}(x) |u_{\delta}^+(x)|^qdx&=&
\int_{\Omega} \phi_{t_q u-s_q\delta^{\frac{2}{2-q}}\eta(\frac{\cdot-a}{\delta})}(x) |t_q u(x)|^qdx\\
&=&t_q^q\int_{\Omega} \phi_u(x)|t_q u(x)-s_q \delta^{\frac{2}{2-q}}\eta(\frac{x-a}{\delta})|^qdx\\
&=&t_q^{2q}\int_{\Omega} \phi_u|u|^q+\delta^{\frac{2q}{2-q}}t_q^q s_q^q\int_{\Omega}\phi_u\eta^q(\frac{x-a}{\delta})dx\\
&=&t_q^{2q}\int_{\Omega} \phi_u|u|^q+\delta^{\frac{2q}{2-q}+3}t_q^q s_q^q \int_{B_2(0)} \phi_u(\delta x+a)\eta^q(x)dx
\end{array}
\end{equation}
and
\begin{equation}\label{eq4.50}
\begin{array}{lll}
\int_{\Omega} \phi_{u_{\delta}}(x) |u_{\delta}^-(x)|^qdx
&=&
\int_{\Omega} \phi_{t_q u-s_q\delta^{\frac{2}{2-q}}\eta(\frac{\cdot-a}{\delta})}(x) |s_q\delta^{\frac{2}{2-q}}\eta(\frac{x-a}{\delta})|^q dx\\
&=&\delta^{\frac{2q}{2-q}}t_q^q s_q^q\int_{\Omega} \phi_u(x)\eta^q(\frac{x-a}{\delta})dx\\
& & +\delta^{\frac{4q}{2-q}}s_q^{2q}\int_{\Omega} \phi_{\eta(\frac{\cdot-a}{\delta})}(x) \eta^q(\frac{x-a}{\delta})dx\\
&=& \delta^{\frac{2q}{2-q}+3}t_q^q s_q^q\int_{B_2(0)} \phi_u(\delta x+a)\eta^q(x)dx\\
& &+\delta^{\frac{4q}{2-q}+3}s_q^{2q}\int_{B_2(0)} \phi_{\eta(\frac{\cdot-a}{\delta})}(\delta x+a) \eta^q(x)dx.\end{array}
\end{equation}

 We claim that for $\delta>0$ small enough, there exist $t_q,s_q\in(0,\infty)$ such that
\begin{equation}\label{eq4.51}
t_q u_{\delta}^+ +s_q u_{\delta}^-\in \mathcal{N}_{nod,q}.
\end{equation}
Indeed,  \eqref{eq4.51} holds
if and only if
\begin{equation}\label{eq4.52}
I_q'(t_q u_{\delta}^+ +s_q u_{\delta}^-)t_q u_{\delta}^+=0
\end{equation}
and
\begin{equation}\label{eq4.53}
I_q'(t_q u_{\delta}^+ +s_q u_{\delta}^-)s_q u_{\delta}^-=0.
\end{equation}
Since $I_q'(u)u=0,$ we deduce from \eqref{eq4.49} and \eqref{eq4.52} that
\begin{equation}\label{eq4.55}\begin{array}{lll}
0&=(t_q^2-t_q^p)\int_{\Omega} |u|^p +(t_q^2- t_q^{2q})\int_{\Omega} \phi_u |u|^q - \delta^{\frac{2q}{2-q}+3}t_q^q s_q^q \int_{B_2(0)} \phi_u(\delta x+a)\eta^q(x)dx\\
&=:G^1(\delta,t_q,s_q).
\end{array}\end{equation}
Moreover, by \eqref{eq4.50} and \eqref{eq4.53}, we have
\begin{equation}\label{eq4.54}\begin{array}{rll}
0&=s_q^2 \int_{B_2(0)}|\nabla \eta|^2
-s_q^p\delta^{\frac{2(p-q)}{2-q}}\int_{B_2(0)} |\eta|^p-t_q^q s_q^q\int_{B_2(0)} \phi_u(\delta x+a)\eta^q(x)dx\\
&\quad-s_q^{2q}\delta^{\frac{2q}{2-q}} \int_{B_2(0)}\phi_{\eta(\frac{\cdot-a}{\delta})}(\delta x+a)\eta^q(x)dx\\
&=:G^2(\delta,t_q,s_q).
\end{array}\end{equation}

Now, we define a new vector function $(\tilde{G}^1,\tilde{G}^2): [0,+\infty)\times (0,+\infty)^2\to \mathbb{R}^2$ by
\begin{equation*}
(\tilde{G}^1(\delta,t_q,s_q),\tilde{G}^2(\delta,t_q,s_q))=
\left\{\begin{array}{lll}
(G^1(\delta,t_q,s_q),G^2(\delta,t_q,s_q))\ &\mbox{if}&\ \delta>0,\\
(H^1(t_q,s_q),H^2(t_q,s_q))\ &\mbox{if}&\ \delta=0,
\end{array}\right.\end{equation*}
where
$$\begin{array}{lll}
H^1(t_q,s_q)&=(t_q^2-t_q^p)\int_{\Omega} |u|^p +(t_q^2-t_q^{2q})\int_{\Omega} \phi_u |u|^q,\\
H^2(t_q,s_q)&=s_q^2 \int_{B_2(0)}|\nabla \eta|^2 -t_q^q s_q^q\phi_u(a)\int_{B_2(0)}\eta^q(x)dx.
\end{array}$$
Since $\eta(\frac{\cdot-a}{\delta})\in C_c^{\infty}(\Omega)$, the standard elliptic regularity argument shows that $\phi_{\eta(\frac{\cdot-a}{\delta})}\in C^2(\overline{\Omega}).$ Note that  $0\leq \eta\leq 1.$  Then it follows that
$$\mbox{$|\phi_{\eta(\frac{\cdot-a}{\delta})}|\leq C_4$ and $|\int_{B_2(0)} \phi_{\eta(\frac{\cdot-a}{\delta})}(\delta x+a) \eta^q(x)dx|\leq C_5$,}$$
for some $C_4,C_5>0$ independent of $\delta\in (0,\tilde{\delta}/2).$
Thus
$$G^1(\delta,t_q,s_q)\to H^1(t_q,s_q)\ \mbox{and}\ G^2(\delta,t_q,s_q)\to H^2(t_q,s_q)\ \mbox{as}\ \delta\to 0.$$
So we can conclude that $(\tilde{G}^1,\tilde{G}^2)$ is continuous in $[0,+\infty)\times (0,+\infty)^2$.
In addition,  $H^1(t_q,s_q)=H^2(t_q,s_q)=0$
has a unique solution
$$(t^*_q,s^*_q)=\left(1,\left(\frac{\phi_u(a)\int_{B_2(0)} |\eta|^q}{\int_{B_2(0)}|\nabla \eta|^2}\right)^{1/(2-q)}\right),$$
and
$$\begin{array}{lll}
\det\left(\begin{array}{lll}
\frac{\partial \tilde{G}^1}{\partial t_q} & \frac{\partial \tilde{G}^1}{\partial s_q} \\
\frac{\partial \tilde{G}^2}{\partial t_q} & \frac{\partial \tilde{G}^2}{\partial s_q}
\end{array}\right)(0,t^*_q,s^*_q)
=\det\left(\begin{array}{lll}
\frac{\partial H^1}{\partial t_q} & \frac{\partial H^1}{\partial s_q} \\
\frac{\partial H^2}{\partial t_q} & \frac{\partial H^2}{\partial s_q}
\end{array}\right)(t^*_q,s^*_q)\\
=(2-q)\left(\int_{\Omega}|\nabla \eta|^2\right)^{\frac{1-q}{2-q}}
\left(\phi_u(a)\int_{\Omega}|\eta|^q\right)^{1/(2-q)}\bigg[(2-p)\int_{\Omega}|u|^p
+(2-2q)\int_{\Omega}\phi_u |u|^q\bigg]\\
\neq 0.
\end{array}$$
Then by applying the implicit function theorem to $(\tilde{G}^1,\tilde{G}^2)$ at $(0,t^*_q,s^*_q)$, we know that for $\delta>0$ small enough, there exists unique $(t_q(\delta),s_q(\delta))$ satisfying \eqref{eq4.55}, \eqref{eq4.54} and
\begin{equation*}
(t_q(\delta),s_q(\delta))\to (t^*_q,s^*_q)\quad \mbox{as $\delta\to 0.$}
\end{equation*}
Hence \eqref{eq4.51} follows  and  the claim holds.

By direct calculations,  we deduce that $t_q(\delta)u_{\delta}^+ +s_q(\delta)u_{\delta}^-\to u$ in $H_0^1(\Omega)$ as $\delta\to 0,$ and
$$\inf_{v\in \mathcal{N}_{nod,q}} I_q(v) \leq \lim\limits_{\delta\to 0} I_q(t_q(\delta)u_{\delta}^+ +s_q(\delta)u_{\delta}^-)=I_q(u)$$
which implies $m_{nod,q}\leq m_{q}.$ This combined with $m_{nod,q}\geq m_{q}$ yields that \begin{equation}\label{eq4.56}
m_{nod,q}=m_{q}.
\end{equation}
If there exists $w\in \mathcal{N}_{nod,q}$ such that $I_q(w)=m_{nod,q},$ in view of the fact that  $\mathcal{N}_{nod,q}\subset \mathcal{N}_q$ and \eqref{eq4.56},  $w$ can be also viewed as a minimizer of $I_q$ over Nehari manifold $\mathcal{N}_q.$  From Theorem \ref{prop} (ii), we infer that either $w>0$ or $w<0,$ which contradicts with the assumption $w\in \mathcal{N}_{nod,q}$.

Therefore, there is no least energy nodal solution for $q\in(1,2).$ The proof is complete.
\hfill{\QEDopen}

\section{$p=2$: linear local perturbtion}
In this section, we are devoted to the proof of Theorems \ref{zhdl-1} and \ref{zhdl-2}. Let $q\in (1,5)$. Without loss of generality, we assume $\mu=1$  in \eqref{eq1.1} and consider the following equation
\begin{equation}\label{xianxing}\left\{\begin{array}{rll}
-\Delta u&=\lambda u+\phi(x)|u|^{q-2}u\\
-\Delta \phi&=|u|^q\\
u&=\phi=0
\end{array}\right.
\begin{gathered}\begin{array}{rll}
&\mbox{in}\ \Omega,\\
&\mbox{in}\ \Omega,\\
&\mbox{on}\ \partial\Omega.
\end{array}\end{gathered}\end{equation}
where $\lambda$ is a constant.
\subsection{Proof of Theorem \ref{zhdl-1}}
For  $\lambda\in(-\infty,\lambda_1)$, it is easy to see that
$$\|u\|:=\int_{\Omega}|\nabla u|^2-\lambda\int_{\Omega}|u|^2$$
is an equivalent norm to the usual one in $H_0^1(\Omega)$. Then by similar arguments of the proof of Theorems \ref{prop} and \ref{theoremmain}, we can prove Theorem \ref{zhdl-1}. Here we omit the details.

\subsection{Proof of Theorem \ref{zhdl-2}}
In order to to prove Theorem \ref{zhdl-2}, we first introduce the method of generalized Nehari manifold, as stated in \cite{Szulkin-Weth} and \cite{Szulkin-Weth2009}.

Let $E$ be a Hilbert space with an orthogonal decomposition
$$E=E_+\oplus E_0 \oplus E_-=E_+\oplus F,$$
where $\operatorname{dim} E_0<\infty.$ Here and hereafter, we write
$$u=u_++u_0+u_-,\quad u_{\pm}\in E_{\pm},\ u_0\in E_0$$
and
$$\hat{E}(u):=\mathbb{R}_0^+u\oplus E_0\oplus E_-.$$

\begin{Proposition}\label{generanehari}(\cite{Szulkin-Weth}[Corollary 33])
Suppose that functional $I\in C^1(E,R)$ satisfies
\item{( $B_1$ )} $I(u)=\frac{1}{2}\|u_+\|-\frac{1}{2}\|u_-\|-\psi(u),$ where $\psi(0)=0$,  $\psi$ is weakly lower semi-continuous, and $\frac{1}{2}\psi'(u)u>\psi(u)>0$ for all $u\neq 0$;
    \item{( $B_2$ )} for any $w\in E_0\oplus E_-$,  there always exists $\hat{m}(w)\in\hat{E}(w)$ such that $\hat{m}(w)$ is a critical point of $I|_{\hat{E}(w)}$. Moreover, $\hat{m}(w)$ is the unique global maximum of $I|_{\hat{E}(w)}$;
    \item{( $B_3$ )} there exists $\delta>0$ such that for any $w\in E\backslash F,$ $\|\hat{m}(w)^+\|\geq \delta$, and for each compact set $W\subset\subset E\backslash F$, there is $C_{W}>0$ such that  $\|\hat{m}(w)\|\leq C_{W},\quad \forall w\in W$.
Let
    $$S^+=\{u\in E_+: \|u\|_E=1\},$$
    $$\Psi: S^+\to \mathbb{R}, \Psi(w):=I(\hat{m}(w))$$
    $$M:=\{u\in E_0\oplus E_-: I'(u)u=0\ \mbox{and}\ I'(u)v=0,\ \forall v\in E_0\oplus E_-\}$$
    and
    $$c=\inf\limits_{M} I(u)$$
    be the least energy level.
        Then the following statements hold.
        \item{(i)} $\Psi\in C^1(S^+, \mathbb{R}),$
      $$\Psi'(w)z=\|m(w)^+\|I'(\hat{m}(w))z,\quad \forall z\in T_w(S^+);$$
       \item{(ii)}
        If $(w_n)_{n\geq 1}\subset S^+$ be a P.S. sequence of $\Psi$, then $\hat{m}(w_n)$ is also a P.S. sequence of $I$. Furthermore, if    $(u_n)_{n\geq 1}\subset M$ is a bounded P.S. sequence of $I$, then $\frac{\hat{m}^{-1}(u_n)}{\|\hat{m}^{-1}(u_n)\|}$ is a P.S. sequence of $\Psi$;
            \item{(iii)} $w\in S^+$ is a critical point of $\Psi$ if and only if $\hat{m}(w)$ is a critical point of $I$; moreover,  $\Psi(w)=I(\hat{m}(w))$ and $\inf_{S^+} \Psi=\inf_{M} I;$
         \item{(iv)} $$c=\inf_M I=\inf\limits_{w\in E_0\oplus E_-}\max\limits_{u\in \hat{E}(w)}I(u)=\inf\limits_{u\in S^+}\max\limits_{u\in \hat{E}(w)}I(u).$$
\end{Proposition}

Now we consider \eqref{xianxing}.
Suppose that there is some $1\leq k<m$ such that $\lambda_k<\lambda=\lambda_{k+1}=\cdots=\lambda_m<\lambda_{m+1}$.  Note that $m$ could be equal to $k$, and in this case we assume $\lambda_k<\lambda<\lambda_{k+1}$.

Let
$$E=E_+\oplus E_0\oplus E_-$$
 be the orthogonal decomposition
corresponding to the spectrum of $-\Delta-\lambda$ in $E$. That is,
$$E_-=span\{e_1,\cdots,e_k\}\quad\mbox{and}\quad E_0=span\{e_{k+1},\cdots,e_m\}.$$
So $u=u_++u_0+u_-\in E_+\oplus E_0\oplus E_-,$ and there is an equivalent norm $\|\cdot\|$ in $E$ such that
$$\int_{\Omega}(|\nabla u|^2-\lambda u^2)dx=\|u_+\|^2-\|u_-\|^2.$$
Thus, the functional of \eqref{xianxing}
$$I_{q}(u)=\frac{1}{2}\int_{\Omega}(|\nabla u|^2-\lambda u^2)dx-\frac{1}{2q}\int_{\Omega}\phi_u |u|^qdx$$
can be written as
$$I_{q}(u)=\frac{1}{2}\|u_+\|^2-\frac{1}{2}\|u_-\|^2-\frac{1}{2q}\int_{\Omega}\phi_u |u|^qdx.$$

Define a generalized Nehari manifold
$$M:=\{u\in H_0^1(\Omega)\backslash\{0\}: I_{q}'(u)u=0, I_{q}'(u)v=0, \forall v\in E_-\}$$
and the infimum
$$c_0:=\inf\limits_{M} I_{q}(u).$$
It is easy to see that when $E_-=\{0\}$,  $M$ is the usual Nehari manifold. Moreover, if $u\neq 0$ satisfies $I_{q}'(u)=0$,  then $I_{q}(u)=I_{q}(u)-\frac{1}{2}I_{q}'(u)u=(\frac{1}{2}-\frac{1}{2q})\int_{\Omega}\phi_u |u|^qdx>0.$
Note that $I_{q}(u)\leq 0$ for $u\in E_0\oplus E_-.$  Then all critical points of $I_{q}$ belong to generalized Nehari manifold $M$. Thus $M$ is a natural generalization of the standard Nehari manifold.

\begin{Lemma}\label{jidazhi}
If $u\in M$,  then for any $0\neq w\in Z:=\{su+v: s\geq -1, v\in E_0\oplus E_-\},$
 $$I_{q}(u+w)<I_{q}(u),$$
that is, $u$ is the unique global maximum of $I_{q}|_{\hat{E}(u)}$.
\end{Lemma}
\begin{proof}
Set $w=su+v$ and $z=(1+s)u+v$,  where $s\geq 1, v\in E_0\oplus E_-.$

Let $a:E\times E\to \mathbb{R}$ be a symmetric bilinear functional defined as
$$a(u_1,u_2)=\int_{\Omega}\nabla u_1\nabla u_2 -\lambda u_1 u_2dx\quad\forall u_1, u_2\in E.$$
In view of $u\in M,$
\begin{equation}\label{xjian}\begin{array}{lll}
&&I_{q}(u+w)-I_{q}(u)\\
&=&\frac{1}{2}[a(u+su+v, u+su+v)-a(u,u)]+\frac{1}{2q}\int_{\Omega}\phi_u |u|^q-\frac{1}{2q}\int_{\Omega}\phi_{u+w} |u+w|^q\nonumber\\
&=&a(u,(\frac{s^2}{2}+s)u+(s+1)v)+\frac{1}{2}a(v,v)+\frac{1}{2q}\int_{\Omega}\phi_u |u|^q-\frac{1}{2q}\int_{\Omega}\phi_{u+w} |u+w|^q\nonumber\\
&=&\int_{\Omega}\phi_u |u|^{q-2}u[(\frac{s^2}{2}+s)u+(s+1)v]+\frac{1}{2}a(v,v)\nonumber\\
&\quad& +\frac{1}{2q}\int_{\Omega}\phi_u |u|^q-\frac{1}{2q}\int_{\Omega}\phi_{u+w} |u+w|^q \tag{\text{1}}\\
&<&\int_{\Omega}\phi_u |u|^{q-2}u[(\frac{s^2}{2}+s)u+(s+1)v+\frac{u}{2}]+\frac{1}{2}a(v,v)-\frac{1}{2q}\int_{\Omega}\phi_{u+w} |u+w|^q\nonumber\\
&=&\int_{\Omega}\phi_u |u|^{q-2}u[-\frac{(1+s)^2}{2}u+(1+s)z]+\frac{1}{2}a(v,v)-\frac{1}{2q}\int_{\Omega}\phi_{u+w} |u+w|^q.\nonumber
\end{array}\end{equation}
Thus, when $uz<0$,
\begin{equation}\label{uzxy0}
I_{q}(u+w)-I_{q}(u)<0.
\end{equation}
When $uz>0$, let
$$\begin{array}{lll}
i(s)&=I_{q}(z)-I_{q}(u)\\
&=\int_{\Omega}\phi_u |u|^{q-2}u[(\frac{s^2}{2}+s)u+(s+1)v]+\frac{1}{2}a(v,v)
+\frac{1}{2q}\phi_u |u|^q-\frac{1}{2q}\phi_{z}|z|^q.
\end{array}$$
It follows from  (1) that
$$i(-1)=(\frac{1}{2q}-\frac{1}{2})\phi_u |u|^q-\frac{1}{2q}\phi_{z}|z|^q<0,$$
and as $s\to +\infty$, $i(s)\to -\infty$.
Then there exists $s_0\in(-1,+\infty)$, a maximum point of $i$, such that
\begin{equation*}\label{}\begin{array}{lll}
0&=i'(s_0)\\
&=\phi_u |u|^{q-2}u[(1+s_0)u+v]-\frac{1}{2q}\phi_{z}|z|^{q-2}zu\\
&=(\phi_u |u|^{q-2}-\frac{1}{2q}\phi_{z}|z|^{q-2})zu.
\end{array}\end{equation*}
Since $uz>0$ and $\phi_t |t|^{q-2}$ is strictly monotone in $t$,
 then $u=z$ and  the maximum of $i$ is $i(s_0)=0.$
So, when $uz>0$,
\begin{equation}\label{uzdy0}
I_{q}(u+w)-I_{q}(u)<0.
\end{equation}
Therefore, from \eqref{uzxy0} and \eqref{uzdy0}, this lemma follows.
\end{proof}

By using Lemma \ref{jidazhi}, we have the following results.
\begin{Lemma}\label{lemgynehari2}
The following statements are true:
\item{(i)} there exist $\alpha>0$ such that $c_0=\inf_M I_{q}(u)\geq \alpha>0$;
\item{(ii)} for any $u\in M,$ there holds $\|u_+\|\geq \sqrt{2c_0}.$
\end{Lemma}
\begin{proof}
(i) For any $u\in M,$ it follows from Lemma \ref{jidazhi} that
$$I_{q}(u)\geq  I_{q}(su^+), \forall\ s>0.$$
Note that
$$I_{q}(su^+)=\frac{1}{2}\|su^+\|^2-\frac{1}{2q}\int_{\Omega}\phi_{su^+} |su^+|^qdx\\
\geq \frac{1}{2}\|su^+\|^2-\frac{C}{2q}\|su^+\|^{2q},$$
where $C$ is independent of $u$ and $s$. Then if $\|su^+\|=(\frac{q}{2C})^{\frac{1}{2(q-1)}}$, we have
\begin{equation}\label{su+}
I_{q}(su^+)\geq \frac{1}{2}(\frac{q}{2C})^{\frac{1}{2(q-1)}}.
\end{equation}
Let $\alpha=\frac{1}{2}(\frac{q}{2C})^{\frac{1}{2(q-1)}},$
then for any $u\in M$, we derive

$$I_{q}(u)\geq \frac{1}{2}(\frac{q}{2C})^{\frac{1}{2(q-1)}}.$$

Thus $c_0=\inf_M I_{q}(u)\geq \alpha>0$ and (i) follows.

(ii) If $u\in M,$ then
$$c_0\leq \frac{1}{2}(\|u_+\|^2-\|u_-\|^2)-\frac{1}{2q}\int_{\Omega}\phi_u |u|^q\leq \frac{1}{2}\|u_+\|^2.$$
Hence we obtain $\|u_+\|\geq \sqrt{2c_0}.$
\end{proof}

\begin{Lemma}\label{lemgynehari3}
For each $u\in E\backslash (E_0\oplus E_-),$ there exists a unique  $\hat{m}(u)\in M$ such that
$$I_{q}(\hat{m}(u))=\max\limits_{v\in\hat{E}(u)} I_{q}(v).$$
\end{Lemma}
\begin{proof}
Given $u\in E\backslash (E_0\oplus E_-),$ since $\dim \hat{E}(u)<+\infty,$ we have
$$\beta:=\inf_{\substack{\|z\|=1\\ z\in \hat{E}(u)}}\int_{\Omega}\phi_z |z|^q >0.$$
Let
\begin{equation}\label{tildec}
\tilde{C}=(\frac{q}{\beta})^{\frac{1}{2(q-1)}}.
\end{equation}
Then for any $w\in \hat{E}(u)$ with $\|w\|\geq 2\tilde{C}$, we derive
\begin{equation}\label{2tildec}\begin{array}{lll}
I_{q}(w)&\leq \frac{1}{2}\|w\|^2-\frac{1}{2q}\int_{\Omega}\phi_w |w|^q\\
&\leq \frac{1}{2}\|w\|^2-\frac{1}{2q}\|w\|^{2q}\inf\limits_{\substack{\|z\|=1\\ z\in \hat{E}(u)}}\int_{\Omega}\phi_z |z|^q\\
&\leq \left(\frac{1}{2}-\frac{2^{2q-3}}{q}\tilde{C}^{2q-2}\beta\right)\|w\|^2\\
&=(2-2^{2q-1})\|w\|^2<0.
\end{array}\end{equation}
This together with Lemma \ref{lemgynehari2} (i), implies that there exists $u^*\in \{w\in \hat{E}(u): \|w\|\leq 2\tilde{C}\}$ such that $I_{q}(u^*)=\max\limits_{\hat{E}(u)}I_{q}(u)>0.$ This shows that  $u^*$ is a critical point of $I_{q}|_{\hat{E}(u)}$, and
$$\langle I_{q}'(u^*),u^*\rangle=\langle I_{q}'(u^*),v\rangle=0,\quad \forall\ v\in E_0\oplus E_-.$$
By Lemma \ref{jidazhi}, it follows that $\hat{m}(u):=u^*$ is the unique global maximum of $ I_{q}|_{\hat{E}(u)}$. The proof is complete.
\end{proof}

\begin{Lemma}\label{lemgynehari4-1}
Let $W\subset E\backslash (E_0\cup E_-)$ be a compact set. Then there exists $R_W>0$ such that  $I_{q}\leq 0$ in $\hat{E}(u)\backslash B_{R_W}(0)$ for any $ u\in W$.
\end{Lemma}
\begin{proof}
We prove it by contradiction. Suppose the conclusion is not true. Then  there exist $u_n\in W$ and $w_n\in \hat{E}(u_n)$ such that $I_{q}(w_n)\geq 0,\ \forall n\geq 1$ and $\|w_n\|\to \infty$ as $n\to\infty.$ In view of $\hat{E}(u_n)=\hat{E}(\frac{(u_n)_+}{\|(u_n)_+\|}),$ without loss of generality, we may assume
$u_n\in E_+$ and $\|u_n\|=1.$ Since $W$ is a compact set, there exists a subsequence of $(u_n)_{n\geq 1}$, still denoted by $(u_n)_{n\geq 1},$ such that $u_n\to u\in E_+$ and $\|u\|=1.$

Let
$$v_n=\frac{w_n}{\|w_n\|}=s_nu_n+(v_n)_0+(v_n)_-.$$
 Then
\begin{equation}\label{quyuwuqiong}
0\leq \frac{I_{q}(w_n)}{\|w_n\|^2}=\frac{1}{2}(s_n^2-\|(v_n)_-\|^2)-\|w_n\|^{2q-2}\int_{\Omega} \phi_{v_n}|v_n|^q.
\end{equation}
So $\|(v_n)_-\|^2\leq s_n^2=\|v_n\|^2-\|(v_n)_-\|^2$ and $\frac{\sqrt{2}}{2}\leq s_n\leq 1.$ Thus, there is a subsequence of $(s_n)_{n\geq 1},$ still denoted by $(s_n)_{n\geq 1},$ such that $s_n\to s \neq 0,$ $v_n\rightharpoonup v$ in $E$, $v_n(x)\to v(x)$ a.e. in $\Omega$.
Hence
$$v=su+v_0+v_-\neq 0, \quad \int_{\Omega}\phi_{v_n}|v_n|^q\to \int_{\Omega}\phi_{v}|v|^q
$$
and
$$\|w_n\|^{2q-2}\int_{\Omega} \phi_{v_n}|v_n|^q\to \infty.$$
Therefore, the right side of \eqref{quyuwuqiong} tends to $-\infty,$ which is contradiction. The proof is complete.
\end{proof}

\begin{Lemma}\label{lemgynehari4}
The map $E \backslash (E_0\oplus E_-)\to M,$  $u\mapsto \hat{m}(u)$ is continuous.
\end{Lemma}
\begin{proof}
Let $(u_n)_{n\geq 1}\subset E\backslash(E_0\oplus E_-)$ be a sequence satisfying $u_n\to u$ in $E\backslash(E_0\oplus E_-)$. In view of $\hat{m}(u)=\hat{m}(u_+)=\hat{m}(\frac{u_+}{\|u_+\|})$, for the sake of convenience, we assume $u_n\in E_+\backslash\{0\}$ and $\|u_n\|=\|u\|=1$.  Then by Lemma \ref{lemgynehari4-1}, there exists some $\bar{R}>0$ such that for $n$ large enough,  $\|\hat{m}(u_n)\|\leq \bar{R}$. Since
$$\hat{m}(u_n)=\|\hat{m}(u_n)_+\|(u_n)_+ +\hat{m}(u_n)_0+\hat{m}(u_n)_-,$$
then there exists subsequence of $(\hat{m}(u_n))$ such that
$$\|\hat{m}(u_n)_+\|\to s, \quad
\hat{m}(u_n)_0+\hat{m}(u_n)_-\to v\in E_0\oplus E_-.$$
 Hence,
$$\lim\limits_{n\to\infty}I_{q}(\hat{m}(u_n))=\lim\limits_{n\to\infty}I_{q}(\|\hat{m}(u_n)_+\|(u_n)_++\hat{m}(u_n)_0+\hat{m}(u_n)_-)=I_{q}(su+v)$$
and
$$\begin{array}{lll}
0=\lim\limits_{n\to\infty}I_{q}'(\hat{m}(u_n))\hat{m}(u_n)=I_{q}'(su+v)(su+v).
\end{array}$$
So $su+v=\hat{m}(u)$ and $\hat{m}(u_n)\to \hat{m}(u).$ The conclusion follows.
\end{proof}

\begin{Lemma}\label{lemgynehari5}
Suppose that $(u_n)_{n\geq 1}\subset M$ is a P.S. sequence of functional $I_{q}$, then there is a convergence subsequence.
\end{Lemma}
\begin{proof}
Let $(u_n)_{n\geq 1}\subset M$ be a P.S. sequence satisfying $I_{q}(u_n)\leq d$ and $I_{q}'(u_n)\to 0$ for some $d>0.$

First, we claim that $(u_n)_{n\geq 1}$ is bounded. In fact, suppose on the contrary that $(u_n)_{n\geq 1}$ is unbounded. Let $v_n:=\frac{u_n}{\|u_n\|}$,  then there exists subsequence such that $\|u_n\|\to\infty$ and $v_n\rightharpoonup v$. Moreover,  $v=0$ and $(v_n)_+\nrightarrow 0$.
Observe
$$0\leq \frac{I_{q}(u_n)}{\|u_n\|^2}=\frac{1}{2}(\|(v_n)_+\|^2-\|(v_n)_-\|^2)-\frac{1}{2q}\|u_n\|^{2q-2}\int_{\Omega}\phi_{v_n}|v_n|^q.
$$
However, if $v\neq 0$, the right side of the inequality tends to $-\infty$ as $n\to\infty,$ which is a contradiction. Hence $v=0.$

If $(v_n)_+\to 0$, since the above inequality implies $\|(v_n)_+\|^2\geq \|(v_n)_-\|^2$, we have $(v_n)_-\to 0$. Thus
$$\|(v_n)_0\|^2=\|v_n\|^2-\|(v_n)_+\|^2-\|(v_n)_-\|^2\to 1,$$
which implies $v\neq 0.$ This is a contradiction. Hence $(v_n)_+\nrightarrow 0$.
Therefore, there exists some $\gamma>0$ such that $\|(v_n)_+\|\geq \gamma>0, \forall\ n\geq 1$.  By Proposition \ref{lem1.1}(iii), it follows that for any $s>0,$
\begin{equation}\begin{array}{lll}
d\geq I_{q}(u_n)\geq I_{q}(s(v_n)_+)
\geq \frac{1}{2}s^2\gamma^2-\frac{1}{2q}s^{2q}\int_{\Omega}\phi_{(v_n)_+}(v_n)_+
\to  \frac{1}{2}s^2\gamma^2.
\end{array}\end{equation}
Clearly, by taking $s=\frac{2\sqrt{d}}{\gamma}$, we get a contradiction . Thus the claim follows.

Furthermore, there is $u\in H_0^1(\Omega)$ such that $u_n\rightharpoonup u$ in $H_0^1(\Omega)$. Then it follows from $I_{q}'(u_n)\to 0$ that $I_{q}'(u)=0$ and
\begin{equation}\label{uundengyu0}
I_{q}'(u)u=0=I_{q}'(u_n)u_n.
\end{equation}
Since
$$\int_{\Omega}|u_n|^2\to \int_{\Omega}|u|^2\quad\mbox{and}\quad \int_{\Omega}\phi_{u_n}|u_n|^q\to\int_{\Omega}\phi_{u}|u|^q,$$
it follows from  \eqref{uundengyu0} that
 $\int_{\Omega}|\nabla u_n|^2\to \int_{\Omega}|u|^2$ and
 $u_n\to u$  $H_0^1(\Omega).$
The proof is complete.
\end{proof}

\textbf{Proof of Theorem \ref{zhdl-2}:}
By a direct computation and Lemmas \ref{lemgynehari5}, \ref{lemgynehari3}, \ref{lemgynehari2} and \ref{lemgynehari4-1},  we conclude $I_{q}\in C^1(E,\mathbb{R})$ satisfies $(P.S.)$ condition and ($B_1$)($B_2$)($B_3$) in Proposition \ref{generanehari}. Then it follows from Lemma \ref{lemgynehari2} and Proposition \ref{generanehari} that there exists a minimizing sequence $(w_n)_{n\geq 1}\subset S^+$ such that
$\Psi(w_n)\to \inf_{S^+}\Psi,$ where $\Psi: S^+\to \mathbb{R},$ $\Psi(v)=I_{q}(\hat{m}(v)).$ By the Ekeland principle, there holds
$\Psi'(w_n)\to 0.$ Hence, by Proposition \ref{generanehari}(ii), we have $u_n:=\hat{m}(w_n)$ is a P.S. sequence of $I_{q}$. This, combined with Lemma \ref{lemgynehari5} and Proposition \ref{generanehari}, shows that there exists a minimizer $w\in S^+$ of $\Psi$. Thus, $u:=\hat{m}(w)$ is the ground state and $I_{q}(u)=c_0,$ which implies $u\neq 0.$

Furthermore, we can show that $u$ is a nodal solution. In fact, if $u$ does not change sign, without loss of generality, we assume $u\geq 0.$ By choosing first eigenfunction $e_1>0$ of $-\Delta$,  we have
\begin{equation*}\begin{array}{lll}
0&=I_{q}'(u)e_1\\
&=\int_{\Omega}\nabla u\nabla e_1-\lambda\int_{\Omega} u e_1-\int_{\Omega}\phi_{u}|u|^{q-2}ue_1\\
&=(\lambda_1-\lambda)\int_{\Omega} u e_1-\int_{\Omega}\phi_{u}|u|^{q-2}ue_1\\
&<0.
\end{array}\end{equation*}
This is a  contradiction.

Therefore, $u$ is a least energy nodal solution, and the proof is complete.

\bigskip
Changfeng Gui

College of Mathematics and Econometrics

Hunan University

Changsha 410082, China

\vskip 0.25cm
Department of Mathematics

University of Texas at San Antonio

San Antonio, TX 78249 USA

E-mail: changfeng.gui@utsa.edu

\bigskip

Hui Guo

College of Mathematics and Computing Science 

Hunan University of Science and Technology 

Xiangtan, Hunan 411201, P. R. China

\vskip 0.25cm

College of Mathematics and Econometrics

Hunan University

Changsha 410082, China

Email:  huiguo\_math@163.com

\medskip

\medskip
\end{document}